\documentclass[reqno, a4paper, 11pt]{amsart}
\usepackage{pictexwd,dcpic,amssymb}
\usepackage{amsmath}

\newtheorem{theorem}{Theorem}[section]
\newtheorem{proposition}[theorem]{Proposition}
\newtheorem{lemma}[theorem]{Lemma}
\newtheorem{corollary}[theorem]{Corollary}

\theoremstyle{definition}

\theoremstyle{remark}
\newtheorem{remark}[theorem]{Remark}

\numberwithin{equation}{section}

\begin{document}

\title{On some arithmetic properties of Siegel functions (II)}

\author{Ho Yun Jung}
\address{Department of Mathematical Sciences, KAIST}
\curraddr{Daejeon 373-1, Korea} \email{DOSAL@kaist.ac.kr}
\thanks{}

\author{Ja Kyung Koo}
\address{Department of Mathematical Sciences, KAIST}
\curraddr{Daejeon 373-1, Korea}
\email{jkkoo@math.kaist.ac.kr}
\thanks{}

\author{Dong Hwa Shin}
\address{Department of Mathematical Sciences, KAIST}
\curraddr{Daejeon 373-1, Korea}
\email{shakur01@kaist.ac.kr}
\thanks{}
%    General info


\keywords{class fields, modular forms and functions, normal bases,
normal $p$-integral bases, Siegel functions.
\newline This research was supported by Basic Science Research Program through the National Research Foundation of Korea
funded by the Ministry of Education, Science and Technology
(2009-0063182).}

\maketitle

\begin{abstract}
Let $K$ be an imaginary quadratic field with discriminant
$d_K\leq-7$. We deal with problems of constructing normal bases
between abelian extensions of $K$ by making use of singular values
of Siegel functions. First, we show that a criterion achieved from
the Frobenius determinant relation enables us to find normal bases
of ring class fields of orders of bounded conductors depending on
$d_K$ over $K$. Next, denoting by $K_{(N)}$ the ray class field
modulo $N$ of $K$ for an integer $N\geq2$ we consider the field
extension $K_{(p^2m)}/K_{(pm)}$ for a prime $p\geq5$ and an integer
$m\geq1$ relatively prime to $p$ and then find normal bases of all
intermediate fields over $K_{(pm)}$ by utilizing Kawamoto's
arguments (\cite{Kawamoto}). And, we further investigate certain
Galois module structure of the field extension
$K_{(p^{n}m)}/K_{(p^{\ell}m)}$ with $n\geq 2\ell$, which would be an
extension of Komatsu's work (\cite{Komatsu}).
\end{abstract}

\maketitle

\section{Introduction}

Let $F$ be a finite Galois extension of a field $L$. Then there
exists a normal basis of $F$ over $L$, namely a basis of the form
$\big\{x^\gamma:\gamma\in\mathrm{Gal}(F/L)\big\}$ for a single
element $x\in F$ by the normal basis theorem (\cite{Waerden}). After
Okada (\cite{Okada}) had constructed normal bases of the ray class
fields over the Gaussian field $\mathbb{Q}(\sqrt{-1})$, several
other people treated the problem of generating normal bases of
abelian extensions of other imaginary quadratic fields by special
values of elliptic functions or elliptic modular functions
(\cite{Chan}, \cite{Komatsu}, \cite{Schertz}, \cite{Taylor}). And,
Jung-Koo-Shin (\cite{J-K-S}) recently found normal bases of ray
class fields over any imaginary quadratic field with discriminant
$\leq-7$ by utilizing Siegel functions.
\par
 Let $K$ be an imaginary quadratic field and $H_\mathcal{O}$ be
the ring class field of the order $\mathcal{O}$ of conductor
$N\geq2$ in $K$. In number theory, ring class fields over imaginary
quadratic fields play an important role in the study of certain
quadratic Diophantine equations. For example, let $n$ be a positive
integer and $H_\mathcal{O}$ be the ring class field of the order
$\mathcal{O}=\mathbb{Z}[\sqrt{-n}]$ in $K=\mathbb{Q}(\sqrt{-n})$. If
$p$ is an odd prime not dividing $n$, then we have the following
assertions:
\begin{eqnarray*}
&&p=x^2+ny^2~\textrm{is solvable for some integers $x$ and
$y$}\\
&\Longleftrightarrow&\textrm{$p$ splits completely in
$H_\mathcal{O}$}\\
&\Longleftrightarrow&\left\{\begin{array}{ll}\textrm{the Legendre symbol}~\big(\tfrac{-n}{p}\big)=1~\textrm{and}\\
f_n(X)\equiv0\pmod{p}~\textrm{has an integer solution}
\end{array}\right.
\end{eqnarray*}
where $f_n(X)$ is the minimal polynomial of a real algebraic integer
$\alpha$ for which $H_\mathcal{O}=K(\alpha)$ (\cite{Cox}). It is a
classical result by the main theorem of complex multiplication that
for any proper fractional $\mathcal{O}$-ideal $\mathfrak{a}$, the
$j$-invariant $j(\mathfrak{a})$ is an algebraic integer and
generates $H_\mathcal{O}$ over $K$ (\cite{Lang} or \cite{Shimura}).
Unlike the classical case, however, Chen-Yui (\cite{C-Y})
constructed a generator of the ring class field of certain conductor
in terms of the singular value of the Thompson series which is a
Hauptmodul for $\Gamma_0(N)$ or $\Gamma_0(N)^\dag$. Here,
$\Gamma_0(N)=\big\{\gamma\in\mathrm{SL}_2(\mathbb{Z}):\gamma\equiv\left(\begin{smallmatrix}*&*\\0&*\end{smallmatrix}
\right)\pmod{N}\big\}$ and $\Gamma_0^\dag(N)$ is the subgroup of
$\mathrm{SL}_2(\mathbb{R})$ generated by $\Gamma_0(N)$ and
$\left(\begin{smallmatrix}0&-1/{\sqrt{N}}\\\sqrt{N}&0\end{smallmatrix}\right)$.
Similarly, Cox-Mckay-Stevenhagen (\cite{C-M-S}) showed that certain
singular value of a Hauptmodul for $\Gamma_0(N)$ or
$\Gamma_0(N)^\dag$ with rational Fourier coefficients generates
$H_\mathcal{O}$ over $K$. Furthermore, Cho-Koo (\cite{C-K}) recently
revisited and extended these results by using the theory of
Shimura's canonical models and his reciprocity law. On the other
hand, as we see in the above example, it is essential to find the
minimal polynomial of $j(\mathcal{O})$ over $K$, namely, the
\textit{class equation} of $\mathcal{O}$ in order to solve such
quadratic equations. Although there are several known algorithms for
finding the class equations (\cite{C-Y}, \cite{Cox}, \cite{K-Y},
\cite{Morain}), we would like to adopt the idea of Gee and
Stevenhagen (\cite{Gee}, \cite{G-S} or \cite{Stevenhagen}) because
we could not claim with the formers that the conjugates of
$j(\mathcal{O})$ form a normal basis of $H_\mathcal{O}$ over $K$.
\par
In this paper we shall first construct a ring class invariant of
$H_\mathcal{O}$ under the condition
\begin{eqnarray}\label{cumbersome}
d_K\leq-43\quad\textrm{and}\quad 2\leq
N\leq\frac{-\sqrt{3}\pi}{\ln\big(1-2.16e^{-\frac{\pi\sqrt{-d_K}}{24}}\big)}
\end{eqnarray}
in terms of singular values of Siegel functions and also
systematically find its minimal polynomial (Theorems \ref{main},
\ref{conjugate} and Remark \ref{example}). And, through a criterion
developed in \cite{J-K-S} we can show that the conjugates of the
ring class invariant form a normal basis of $H_\mathcal{O}$ over $K$
(Theorem \ref{normal1}). In Section \ref{section5}, however, we will
show without assuming (\ref{cumbersome}) that certain quotient of
singular values of the $\Delta$-function becomes a ring class
invariant, when the conductor of the extension $H_\mathcal{O}/K$ is
a prime power (Theorem \ref{maindelta}).
\par
Next, we shall consider in Section \ref{section6} the extension
$K_{(p^2m)}/K_{(pm)}$ for a prime $p\geq5$ and an integer $m\geq1$
relatively prime to $p$ and, by means of Kawamoto's arguments
(\cite{Kawamoto}), construct a normal basis of $F$ over $K_{(pm)}$
for each intermediate field $F$ via singular values of Siegel
functions as algebraic integers (Theorems \ref{main1} and
\ref{main2}). And, we shall further discuss in Section
\ref{section7} certain Galois module structure of the ring of
$p$-integers of $K_{(p^nm)}$ over that of $K_{(p^\ell m)}$ where $n$
and $\ell$ are positive integers with $n\geq2\ell$, which is
motivated by a relation between the existence of normal basis in
$\mathbb{Z}_p$-extension and the Greenberg's conjecture (\cite{F-N},
\cite{F-K}).

\section{Field of modular functions}

In this section we briefly review some necessary arithmetic
properties of Siegel functions as modular functions.
\par
For a positive integer $N$, let $\zeta_N=e^\frac{2\pi i}{N}$ and
$\mathcal{F}_N$ be the field of modular functions of level $N$ which
are defined over $\mathbb{Q}(\zeta_N)$. Then $\mathcal{F}_N$ is a
Galois extension of $\mathcal{F}_1=\mathbb{Q}\big(j(\tau)\big)$
($j$=the elliptic modular function) whose Galois group is isomorphic
to
$\mathrm{GL}_2(\mathbb{Z}/N\mathbb{Z})/\big\{\pm\left(\begin{smallmatrix}1&0\\0&1\end{smallmatrix}\right)\big\}$.
In order to describe the Galois action on the field $\mathcal{F}_N$
we consider the decomposition of the group
\begin{equation*}
\mathrm{GL}_2(\mathbb{Z}/N\mathbb{Z})\big/\bigg\{\pm
\begin{pmatrix}1&0\\0&1\end{pmatrix}
\bigg\}=\bigg\{\begin{pmatrix}1&0\\0&d\end{pmatrix}
~:~d\in(\mathbb{Z}/N\mathbb{Z})^*\bigg\}\cdot
\mathrm{SL}_2(\mathbb{Z}/N\mathbb{Z})\big/\bigg\{\pm
\begin{pmatrix}1&0\\0&1\end{pmatrix}\bigg\}.
\end{equation*}
Here, the matrix
$\left(\begin{smallmatrix}1&0\\0&d\end{smallmatrix}\right)$ acts on
$\sum_{n=-\infty}^\infty c_n e^\frac{2\pi
in\tau}{N}\in\mathcal{F}_N$ by
\begin{equation}\label{first}
\sum_{n=-\infty}^\infty c_ne^\frac{2\pi in\tau}{N}\mapsto
\sum_{n=-\infty}^\infty c_n^{\sigma_d}e^\frac{2\pi in\tau}{N}
\end{equation}
where $\sigma_d$ is the automorphism of $\mathbb{Q}(\zeta_N)$
induced by $\zeta_N\mapsto\zeta_N^d$. And, for an element
$\gamma\in\mathrm{SL}_2(\mathbb{Z}/N\mathbb{Z})/\big\{\pm
\left(\begin{smallmatrix}1&0\\0&1\end{smallmatrix}\right)\big\}$ let
$\gamma'\in\mathrm{SL}_2(\mathbb{Z})$ be a preimage of $\gamma$ via
the natural surjection
$\mathrm{SL}_2(\mathbb{Z})\rightarrow\mathrm{SL}_2(\mathbb{Z}/N\mathbb{Z})/\big\{\pm\left(\begin{smallmatrix}1&0\\0&1\end{smallmatrix}
\right)\big\}$. Then $\gamma$ acts on $h\in\mathcal{F}_N$ by
composition
\begin{equation}\label{second}
h\mapsto h\circ\gamma'
\end{equation}
as linear fractional transformation (\cite{Lang} or \cite{Shimura}).
\par
For any pair $(r_1,~r_2)\in\mathbb{Q}^2\setminus\mathbb{Z}^2$ we
define a \textit{Siegel function} $g_{(r_1,~r_2)}(\tau)$ on
$\mathfrak{H}$ (=the complex upper half plane) by the following
Fourier expansion
\begin{eqnarray}\label{FourierSiegel}
g_{(r_1,~r_2)}(\tau)=-q_\tau^{\frac{1}{2}\mathbf{B}_2(r_1)}e^{\pi
ir_2(r_1-1)}(1-q_z)\prod_{n=1}^{\infty}(1-q_\tau^nq_z)(1-q_\tau^nq_z^{-1})
\end{eqnarray}
where $\mathbf{B}_2(X)=X^2-X+\frac{1}{6}$ is the second Bernoulli
polynomial, $q_\tau=e^{2\pi i\tau}$ and $q_z=e^{2\pi iz}$ with
$z=r_1\tau+r_2$. Then it is a modular unit which has no zeros and
poles on $\mathfrak{H}$ (\cite{K-L}). For later use we introduce
some arithmetic properties and a modularity condition of Siegel
functions:

\begin{proposition}\label{transformation}
Let $r=(r_1,~r_2)\in\mathbb{Q}^2\setminus\mathbb{Z}^2$. Then
\begin{itemize}
\item[(i)] $g_r(\tau)$ is integral over
$\mathbb{Z}[j(\tau)]$.
\item[(ii)] Let $N$ be the smallest positive integer with
$Nr\in\mathbb{Z}^2$. If $N$ has at least two prime factors, then
$1/g_r(\tau)$ is integral over $\mathbb{Z}[j(\tau)]$. If $N=p^s$ is
a prime power, then $1/g_r(\tau)$ is integral over
$\mathbb{Z}[\frac{1}{p}][j(\tau)]$.
\item[(iii)] For $\gamma\in\mathrm{SL}_2(\mathbb{Z})$ we get
\begin{equation*}
g_{r}^{12}(\tau)\circ\gamma=g_{r\gamma}^{12}(\tau).
\end{equation*}
\item[(iv)] For $s=(s_1,~s_2)\in\mathbb{Z}^2$ we have
\begin{equation*}
g_{r+s}(\tau)=(-1)^{s_1s_2+s_1+s_2}e^{-\pi
i(s_1r_2-s_2r_1)}g_{r}(\tau).
\end{equation*}
\end{itemize}
\end{proposition}
\begin{proof}
See \cite{K-S} Section 3 and Proposition 2.4.
\end{proof}

\begin{proposition}\label{modularity}
Let $N\geq2$. Let
$\big\{m(r)\big\}_{r\in\frac{1}{N}\mathbb{Z}^2\setminus\mathbb{Z}^2}$
be a family of integers such that $m(r)=0$ except finitely many $r$.
Then a product of Siegel functions
\begin{equation*}
\prod_{r\in\frac{1}{N}\mathbb{Z}^2\setminus\mathbb{Z}^2}g_r^{m(r)}(\tau)
\end{equation*}
belongs to $\mathcal{F}_N$, if $\big\{m(r)\big\}$ satisfies
\begin{eqnarray*}
&&\textstyle\sum_r m(r)(Nr_1)^2\equiv\textstyle\sum_r
m(r)(Nr_2)^2\equiv0\pmod{\gcd(2,~N)\cdot N}\\
&&\textstyle\sum_r m(r)(Nr_1)(Nr_2)\equiv0\pmod{N}\\
&&\gcd(12,~N)\cdot\textstyle\sum_rm(r)\equiv0\pmod{12}.
\end{eqnarray*}
\end{proposition}
\begin{proof}
See \cite{K-L} Chapter 3 Theorems 5.2 and 5.3.
\end{proof}

\begin{corollary}\label{F_N}
Let $N\geq 2$. For $r=(r_1,~r_2)\in\frac{1}{N}\mathbb{Z}^2\setminus
\mathbb{Z}^2$ the function $g_{r}^\frac{12N}{\gcd(6,~N)}(\tau)$
satisfies
\begin{equation*}
g_{(r_1,~r_2)}^\frac{12N}{\gcd(6,~N)}(\tau)=g_{(-r_1,~-r_2)}^\frac{12N}{\gcd(6,~N)}(\tau)=g_{(\langle
r_1\rangle,~\langle r_2\rangle)}^\frac{12N}{\gcd(6,~N)}(\tau)
\end{equation*}
where $\langle X\rangle$ is the fractional part of $X\in\mathbb{R}$
so that $0\leq \langle X\rangle<1$. It belongs to $\mathcal{F}_N$.
And, $\gamma$ in $\mathrm{GL}_2(\mathbb{Z}/N\mathbb{Z})/\big\{
\pm\left(\begin{smallmatrix}1&0\\0&1\end{smallmatrix}\right)\big\}\cong\mathrm{Gal}(\mathcal{F}_N/
\mathcal{F}_1)$ acts on the function by
\begin{equation*}
\bigg(g_{r}^\frac{12N}{\gcd(6,~N)}(\tau)\bigg)^\gamma=
g_{r\gamma}^\frac{12N}{\gcd(6,~N)}(\tau).
\end{equation*}
\end{corollary}
\begin{proof}
It is a direct consequence of Propositions \ref{transformation},
\ref{modularity} and definition (\ref{FourierSiegel}).
\end{proof}

\section{Action of Galois groups}

We shall investigate an algorithm for finding all conjugates of the
singular value of a modular function, from which we can determine
the conjugates of the singular values of certain Siegel functions.
\par
Let $K(\neq\mathbb{Q}(\sqrt{-1}),\mathbb{Q}(\sqrt{-3}))$ be an
imaginary quadratic field of discriminant $d_K$ and define
\begin{eqnarray}\label{theta}
\theta=\left\{\begin{array}{ll}\frac{\sqrt{d_K}}{2}&\textrm{for}~d_K\equiv0\pmod{4}\\
\frac{-1+\sqrt{d_K}}{2}&\textrm{for}~
d_K\equiv1\pmod{4}\end{array}\right.
\end{eqnarray}
which is a generator of the ring of integers $\mathcal{O}_K$ of $K$,
that is, $\mathcal{O}_K=\mathbb{Z}[\theta]$. We denote by $H$ the
Hilbert class field. Utilizing the Shimura's reciprocity law Gee and
Stevenhagen (\cite{Gee}, \cite{G-S}) described the actions of
$\mathrm{Gal}(K_{(N)}/H)$ and $\mathrm{Gal}(H/K)$ explicitly. By
extending their idea we shall examine
$\mathrm{Gal}(H_\mathcal{O}/K)$ for the order $\mathcal{O}$ of
conductor $N$.
\par
Under the properly equivalent relation primitive positive definite
quadratic forms $aX^2+bXY+cY^2$ of discriminant $d_K$ determine a
group $\mathrm{C}(d_K)$, called the \textit{form class group of
discriminant $d_K$}. We identify $\mathrm{C}(d_K)$ with the set of
all \textit{reduced} primitive positive definite quadratic forms,
which are characterized by the conditions
\begin{equation}\label{reduced}
-a<b\leq a<c\quad\textrm{or}\quad 0\leq b\leq a=c
\end{equation}
together with the discriminant relation
\begin{equation}\label{disc}
b^2-4ac=d_K.
\end{equation}
Then from the above two conditions for reduced quadratic forms one
can deduce
\begin{equation}\label{bound a}
1\leq a\leq\sqrt{\tfrac{-d_K}{3}}.
\end{equation}
And, for a reduced quadratic form
$Q=aX^2+bXY+cY^2\in\mathrm{C}(d_K)$ we define a CM-point $\theta_Q$
by
\begin{equation}\label{theta_Q}
\theta_Q=\frac{-b+\sqrt{d_K}}{2a}.
\end{equation}
Furthermore, we define
$\beta_Q=(\beta_p)_p\in\prod_{p~:~\textrm{prime}}\mathrm{GL}_2(\mathbb{Z}_p)$
as
\begin{eqnarray}\label{u1}
\beta_p=\left\{\begin{array}{ll}
\left(\begin{smallmatrix}a&\frac{b}{2}\\0&1\end{smallmatrix}\right)&\textrm{if}~p\nmid a\\
\left(\begin{smallmatrix}-\frac{b}{2}&-c\\1&0\end{smallmatrix}\right)&\textrm{if}~p\mid a~\textrm{and}~p\nmid c\\
\left(\begin{smallmatrix}-\frac{b}{2}-a&-\frac{b}{2}-c\\1&-1\end{smallmatrix}\right)&\textrm{if}~
p\mid a~\textrm{and}~p\mid c
\end{array}\right.\qquad\textrm{for}~d_K\equiv0\pmod{4}
\end{eqnarray}
and
\begin{eqnarray}\label{u2}
\beta_p=\left\{\begin{array}{ll}
\left(\begin{smallmatrix}a&\frac{b-1}{2}\\0&1\end{smallmatrix}\right)&\textrm{if}~p\nmid a\\
\left(\begin{smallmatrix}\frac{-b-1}{2}&-c\\1&0\end{smallmatrix}\right)&\textrm{if}~p\mid a~\textrm{and}~p\nmid c\\
\left(\begin{smallmatrix}\frac{-b-1}{2}-a&\frac{1-b}{2}-c\\1&-1\end{smallmatrix}\right)&\textrm{if}~
p\mid a~\textrm{and}~p\mid c
\end{array}\right.\qquad\textrm{for}~d_K\equiv1\pmod{4}.
\end{eqnarray}
 It is then well-known that
$\mathrm{C}(d_K)$ is isomorphic to $\mathrm{Gal}(H/K)$ and the
action of $Q$ on $H$ can be extended to that on $K_{(N)}$ as
\begin{eqnarray}\label{K_N/K}
\mathrm{Gal}(H/K)\cong\mathrm{C}(d_K)&\longrightarrow&\mathrm{Gal}(K_{(N)}/K)\\
Q&\mapsto&\bigg(h(\theta)\mapsto
h^{\beta_Q}(\theta_Q)\bigg)\nonumber
\end{eqnarray}
where $h$ is an element of $\mathcal{F}_N$, defined and finite at
$\theta$. Note that the map (\ref{K_N/K}) is not a homomorphism,
just an injective map. And, observe that
\begin{equation}\label{K_N}
K_{(N)}=K\big(h(\theta)~:~h\in\mathcal{F}_N~\textrm{is defined and
finite at}~\theta\big)
\end{equation}
by the main theorem of complex multiplication (\cite{Lang} or
\cite{Shimura}) and there exists
$\beta\in\mathrm{GL}_2^+(\mathbb{Q})\cap \mathrm{M}_2(\mathbb{Z})$
such that $\beta\equiv \beta_p\pmod{N\mathbb{Z}_p}$ for all primes
$p$ dividing $N$ by the Chinese remainder theorem. Thus the action
of $\beta_Q$ on $\mathcal{F}_N$ is understood as that of $\beta$
which is an element of $\mathrm{GL}_2(\mathbb{Z}/N\mathbb{Z})/\big\{
\pm\left(\begin{smallmatrix}1&0\\0&1\end{smallmatrix}\right)\big\}$
(\cite{Shimura}, \cite{Gee} or \cite{G-S}$)$.
\par
Let
\begin{equation*}
\min(\theta,~\mathbb{Q})=X^2+B_\theta
X+C_\theta=\left\{\begin{array}{ll}X^2-\frac{d_K}{4} &
 \textrm{for}~d_K\equiv0\pmod4\\
 X^2+X+\frac{1-d_K}{4} & \textrm{for}~
 d_K\equiv1\pmod4.\end{array}\right.
\end{equation*}
By the Shimura's reciprocity law we have an isomorphism
\begin{eqnarray}\label{K_N/H}
W_{N,~\theta}\big/\bigg\{\pm\begin{pmatrix}1&0\\0&1\end{pmatrix}\bigg\}
&\stackrel{\sim}{\longrightarrow}&\mathrm{Gal}(K_{(N)}/H)\\
\gamma&\mapsto&\bigg(h(\theta)\mapsto
h^\gamma(\theta)\bigg)\nonumber
\end{eqnarray}
where $h\in\mathcal{F}_N$ is defined and finite at $\theta$, and
\begin{equation*}
W_{N,~\theta}= \bigg\{\begin{pmatrix}t-B_\theta s & -C_\theta s\\s &
t\end{pmatrix}\in\mathrm{GL}_2(\mathbb{Z}/N\mathbb{Z})~:~t,~s\in\mathbb{Z}/N\mathbb{Z}\bigg\}\big/\bigg\{\pm
\begin{pmatrix}1&0\\0&1\end{pmatrix}
\bigg\}
\end{equation*}
(\cite{Shimura}, \cite{Gee} or \cite{G-S}).

\begin{lemma}\label{jfix}
Let $N\geq2$. If the function $j(N\tau)$ satisfies
$j(N\theta)=j(N\tau)\circ\alpha(\theta)$ for some
$\alpha=\left(\begin{smallmatrix}x&y\\z&w\end{smallmatrix}\right)\in\mathrm{SL}_2(\mathbb{Z})$,
then $z\equiv0\pmod{N}$, that is, $\alpha\in\Gamma_0(N)$.
\end{lemma}
\begin{proof}
See \cite{K-S} Lemma 9.2.
\end{proof}

\begin{theorem}\label{H_O/H}
Let $\mathcal{O}$ be the order of conductor $N\geq2$ in $K$. Then we
obtain
\begin{equation*}
\mathrm{Gal}(H_\mathcal{O}/H)\cong
W_{N,~\theta}\big/\bigg\{\begin{pmatrix}t & 0\\0
&t\end{pmatrix}~:~t\in(\mathbb{Z}/N\mathbb{Z})^*\bigg\}.
\end{equation*}
\end{theorem}
\begin{proof}
As is well-known, $H_\mathcal{O}=K\big(j(N\theta)\big)$ (\cite{Lang}
or \cite{Shimura}). Let $\gamma$ be an element of $W_{N,~\theta}$
which is of the form
$\left(\begin{smallmatrix}t&0\\0&t\end{smallmatrix}\right)$ for some
$t\in(\mathbb{Z}/N\mathbb{Z})^*$. If we decompose $\gamma$ into
$\left(\begin{smallmatrix}1&0\\0&d\end{smallmatrix}\right)\beta$ for
some $d\in(\mathbb{Z}/N\mathbb{Z})^*$ and
$\beta\in\mathrm{SL}_2(\mathbb{Z})$, then we obviously achieve
$\beta\in\Gamma_0(N)$. Since the function $j(N\tau)$ is a modular
function for $\Gamma_0(N)$ with rational Fourier coefficients, we
deduce by (\ref{K_N/H}) that
\begin{equation*}
j(N\theta)^\gamma=j(N\tau)^\gamma(\theta)=j(N\tau)^\beta(\theta)=j(N\tau)\circ\beta(\theta)=j(N\theta).
\end{equation*}
\par
Conversely, assume that an element
$\gamma=\left(\begin{smallmatrix}t-B_\theta s&-C_\theta
s\\s&t\end{smallmatrix}\right)$ in $W_{N,~\theta}$ fixes
$j(N\theta)$. Decompose $\gamma$ into
$\left(\begin{smallmatrix}1&0\\0&d\end{smallmatrix}\right)\beta$ for
some $d\in(\mathbb{Z}/N\mathbb{Z})^*$ and
$\beta\in\mathrm{SL}_2(\mathbb{Z})$. By the same reasoning as above
we derive $j(N\theta)^\gamma=j(N\tau)\circ\beta(\theta)$. On the
other hand, we know $\beta\in\Gamma_0(N)$ by Lemma \ref{jfix}, and
so $s\equiv0\pmod{N}$. Therefore
$\gamma=\left(\begin{smallmatrix}t&0\\0&t\end{smallmatrix}\right)$.
This proves the theorem.
\end{proof}

\begin{remark}\label{remark3.3}
We have the degree formula
\begin{equation}\label{degree}
[K_{(N)}:H]=\frac{\phi(N\mathcal{O}_K)w(N\mathcal{O}_K)}{w_K}
\end{equation}
where $\phi$ is the Euler function for ideals, namely
\begin{equation*}
\phi(\mathfrak{p}^n)=\big(\mathbf{N}_{K/\mathbb{Q}}(\mathfrak{p})-1\big)
\mathbf{N}_{K/\mathbb{Q}}(\mathfrak{p})^{n-1}
\end{equation*}
for a power of prime ideal $\mathfrak{p}$, $w(N\mathcal{O}_K)$ is
the number of roots of unity in $K$ which are
$\equiv1\pmod{N\mathcal{O}_K}$ and $w_K$ is the number of roots of
unity in $K$ (\cite{Lang2} Chapter VI Theorem 1). And, for the order
$\mathcal{O}$ of conductor $N$ we know the formula
\begin{equation*}
[H_\mathcal{O}:H]=\frac{N}{[\mathcal{O}_K^*:\mathcal{O}^*]}\prod_{p|N}\bigg(1-\bigg(\frac{d_K}{p}\bigg)\frac{1}{p}\bigg)
\end{equation*}
where $\big(\frac{d_K}{p}\big)$ is the Legendre symbol for an odd
prime $p$ and $\big(\frac{d_K}{2}\big)$ is the Kronecker symbol
(\cite{Cox} Chapter 2 Theorem 7.24). Thus the second part of the
proof depending on Lemma \ref{jfix} can be also established by
showing that
\begin{equation*}
[K_{(N)}:H_\mathcal{O}]=\frac{[K_{(N)}:H]}{[H_\mathcal{O}:H]}=\#\bigg\{
\begin{pmatrix}t&0\\0&t\end{pmatrix}~:~t\in(\mathbb{Z}/N\mathbb{Z})^*\bigg\}
\big/\bigg\{ \pm\begin{pmatrix}1&0\\0&1\end{pmatrix}\bigg\}.
\end{equation*}
\end{remark}

\begin{theorem}\label{conjugate}
Let $\mathcal{O}$ be the order of conductor $N\geq2$ in $K$ and $f$
be an element of $\mathcal{F}_N$ such that $f(\theta)\in
H_\mathcal{O}$. Then
\begin{equation*}
\bigg\{f^{\gamma\cdot \beta_Q}(\theta_Q)~:~\gamma\in
W_{N,~\theta}\big/\bigg\{\begin{pmatrix}t&0\\0&t\end{pmatrix}~:~t\in(\mathbb{Z}/N\mathbb{Z})^*\bigg\}~\textrm{and}~Q\in
\mathrm{C}(d_K)\bigg\}
\end{equation*} is the set of all conjugates
of $f(\theta)$ under the action of $\mathrm{Gal}(H_\mathcal{O}/K)$.
\end{theorem}
\begin{proof}
The assertion follows from the following diagram:
\begin{equation*}
\begindc{0}[3]

\obj(0,0)[A]{$K$}

\obj(0,15)[B]{$H$}

\obj(0,30)[C]{$H_\mathcal{O}$}

\obj(0, 35)[D]{\underline{Fields}}

\obj(30, 35)[E]{\underline{Galois groups}}
\mor{A}{B}{$\quad\Bigg)\quad\mathrm{Gal}(H/K)=\bigg\{\bigg(h(\theta)\mapsto
h^{\beta_Q}(\theta_Q)\bigg)\bigg|_H~:~Q\in
\mathrm{C}(d_K)\bigg\}\quad\textrm{by
(\ref{K_N/K})}$}[\atright,\solidline]

\mor{B}{C}{$\quad\Bigg)\quad\mathrm{Gal}(H_\mathcal{O}/H)\cong
W_{N,~\theta}\big/\bigg\{\begin{pmatrix}t & 0\\0
&t\end{pmatrix}~:~t\in(\mathbb{Z}/N\mathbb{Z})^*\bigg\}\quad\textrm{by
Theorem \ref{H_O/H}}$}[\atright,\solidline]
\enddc
\end{equation*}
where $h$ is an element of $\mathcal{F}_N$, defined and finite at
$\theta$.
\end{proof}

\begin{remark}
Theorem \ref{conjugate} and transformation formulas in Corollary
\ref{F_N} enable us to find all conjugates of the singular value
$\prod_{\tiny\begin{smallmatrix}1\leq
w\leq\frac{N}{2}\\\gcd(w,~N)=1\end{smallmatrix}}g_{(0,~\frac{w}{N})}^\frac{12N}{\gcd(6,~N)}(\theta)$,
which will be used to prove our first main theorem.
\end{remark}

\section{Normal bases of ring class fields}\label{section4}

Let $K(\neq\mathbb{Q}(\sqrt{-1}),~\mathbb{Q}(\sqrt{-3}))$ be an
imaginary quadratic field with $d_K(\leq-7)$ and let $\theta$ be
defined as in (\ref{theta}) and $N\geq2$. If we put
\begin{equation*}
D=\sqrt{\tfrac{-d_K}{3}}\quad\textrm{and}\quad A=|e^{2\pi
i\theta}|=e^{-\pi\sqrt{-d_K}},
\end{equation*}
then $A^\frac{1}{D}=e^{-\sqrt{3}\pi}$ which is independent of $K$.

\begin{lemma}\label{ineq}
We have the following inequalities:
\begin{itemize}
\item[(i)] $\big|\frac{1-\zeta_N}{1-A^\frac{1}{DN}}\big|>1$.
\item[(ii)] $\frac{1}{1-A^{\frac{X}{D}}}<1+A^\frac{X}{1.03D}$
for all $X\geq\frac{1}{2}$.
\item[(iii)]
$\frac{1}{1-A^X}<1+A^{\frac{X}{1.03}}$ for all $X\geq\frac{1}{2}$.
\item[(iv)] $1+X<e^X$ for all $X>0$.
\end{itemize}
\end{lemma}
\begin{proof}
See \cite{J-K-S1} Lemma 4.1.
\end{proof}

\begin{lemma}\label{newlemma1}
Assume the condition
\begin{eqnarray}\label{N,d_K}
d_K\leq-43\quad\textrm{and}\quad 2\leq
N\leq\frac{-\sqrt{3}\pi}{\ln\big(1-2.16e^{-\frac{\pi\sqrt{-d_K}}{24}}\big)}.
\end{eqnarray}
Let $Q=aX^2+bXY+cY^2$ be a reduced primitive positive definite
quadratic form of discriminant $d_K$ and $\theta_Q$ be as in
(\ref{theta_Q}). If $a\geq2$, then the inequality
\begin{equation*}
\bigg|\frac{g_{(0,~\frac{w}{N})}\big(\theta\big)}{g_{(\frac{s}{N},~\frac{t}{N})}(\theta_Q)}\bigg|<1
\end{equation*}
holds for $w\in\mathbb{Z}\setminus N\mathbb{Z}$ and
$(s,~t)\in\mathbb{Z}^2\setminus N\mathbb{Z}^2$.
\end{lemma}
\begin{proof}
We may assume $0\leq s\leq\tfrac{N}{2}$ by Corollary \ref{F_N}. And,
we have $2\leq a\leq D$ by (\ref{bound a}) because $Q$ is a reduced
primitive positive definite quadratic form. From the definition
(\ref{FourierSiegel}) we obtain that
\begin{eqnarray*}
\bigg|\frac{g_{(0,~\frac{w}{N})}\big(\theta\big)}{g_{(\frac{s}{N},~\frac{t}{N})}(\theta_Q)}\bigg|
\leq
A^{\frac{1}{2}\big(\mathbf{B}_2(0)-\frac{1}{a}\mathbf{B}_2(\frac{s}{N})\big)}
\bigg|\frac{1-\zeta_N^w}{1-e^{2\pi
i(\frac{s}{N}\cdot\tfrac{-b+\sqrt{d_K}}{2a}+\frac{t}{N})}}\bigg|\prod_{n=1}^\infty
\frac{(1+A^n)^2}{(1-A^{\frac{1}{a}(n+\frac{s}{N})})(1-A^{\frac{1}{a}(n-\frac{s}{N})})}.
\end{eqnarray*}
Now we see from the fact $a\leq D$ and Lemma \ref{ineq}(i) that
$\big|1-\zeta_N^w\big|<2$ and
\begin{eqnarray*}
\big|1-e^{2\pi
i(\frac{s}{N}\cdot\tfrac{-b+\sqrt{d_K}}{2a}+\frac{t}{N})}\big|&\geq&
\left\{\begin{array}{lll}\big|1-\zeta_N^t\big| &\geq
\big|1-\zeta_N\big| & \textrm{if
$s=0$}\vspace{0.1cm}\\\big|1-A^\frac{s}{Na}\big| & \geq
\big|1-A^\frac{1}{ND}\big| &\textrm{if $s\neq0$}
\end{array}\right.\\
&\geq& 1-A^\frac{1}{ND}.
\end{eqnarray*}
Therefore we achieve that
\begin{eqnarray*}
&&\bigg|\frac{g_{(0,~\frac{w}{N})}\big(\theta\big)}{g_{(\frac{s}{N},~\frac{t}{N})}(\theta_Q)}\bigg|<
\frac{2A^{\frac{1}{2}\big(\mathbf{B}_2(0)-\frac{1}{2}\mathbf{B}_2(0)\big)}}{1-A^\frac{1}{ND}}\prod_{n=1}^\infty\frac{(1+A^{n})^2}
{(1-A^{\frac{n}{D}})(1-A^{\frac{1}{D}(n-\frac{1}{2})})}\\
&&\hspace{9cm}\textrm{by the facts $2\leq a\leq D$, $0\leq s\leq\tfrac{N}{2}$}\\
&\stackrel{}{<}&
\frac{2A^{\frac{1}{24}}}{1-A^\frac{1}{ND}}\prod_{n=1}^\infty
(1+A^{n})^2(1+A^{\frac{n}{1.03D}})
(1+A^{\frac{1}{1.03D}(n-\frac{1}{2})})\quad\textrm{by Lemma \ref{ineq}(ii)}\\
&\stackrel{}{<}&
\frac{2A^{\frac{1}{24}}}{1-A^\frac{1}{ND}}\prod_{n=1}^\infty
e^{2A^{n}+A^{\frac{n}{1.03D}}+A^{\frac{1}{1.03D}(n-\frac{1}{2})}}
\quad\textrm{by Lemma \ref{ineq}(iv)}\\
&\stackrel{}{=}&
\frac{2A^{\frac{1}{24}}}{1-A^\frac{1}{ND}}e^{\frac{2A}{1-A}+
\frac{A^{\frac{1}{1.03D}}+A^{\frac{1}{2.06D}}}{1-A^{\frac{1}{1.03D}}}}\leq
\frac{2e^{-\frac{\pi\sqrt{-d_K}}{24}}}{1-e^{-\frac{\sqrt{3}\pi}{N}}}
~e^{\frac{2e^{-\sqrt{43}\pi}}{1-e^{-\sqrt{43}\pi}}+
\frac{e^{-\frac{\sqrt{3}\pi}{1.03}}+e^{-\frac{\sqrt{3}\pi}{2.06}}}{1-e^{-\frac{\sqrt{3}\pi}{1.03}}}}
\quad\textrm{by the fact $d_K\leq-43$}\\
&\stackrel{}{<}&
\frac{2.16e^{-\frac{\pi\sqrt{-d_K}}{24}}}{1-e^{-\frac{\sqrt{3}\pi}{N}}}<1
\quad\textrm{by the condition (\ref{N,d_K})}.
\end{eqnarray*}
This proves the lemma.
\end{proof}

\begin{lemma}\label{newlemma2}
Assume the condition
\begin{equation}\label{condition2}
d_K\leq-43\quad\textrm{and}\quad2\leq N\leq\sqrt{-d_K}.
\end{equation}
Let $Q=X^2+bXY+cY^2$ be a reduced primitive positive definite
quadratic form of discriminant $d_K$. Then we get the inequality
\begin{equation*}
\bigg|\frac
{g_{(0,~\frac{w}{N})}(\theta)}{g_{(\frac{s}{N},~\frac{t}{N})}(\theta_Q)}\bigg|<
1
\end{equation*}
for $w\in\mathbb{Z}\setminus N\mathbb{Z}$ and
$(s,~t)\in\mathbb{Z}^2\setminus N\mathbb{Z}^2$ with
$s\not\equiv0\pmod{N}$.
\end{lemma}
\begin{proof}
We may assume $1\leq s\leq\tfrac{N}{2}$ by Corollary \ref{F_N}. Then
we establish that
\begin{eqnarray*}
&&\bigg|\frac
{g_{(0,~\frac{w}{N})}(\theta)}{g_{(\frac{s}{N},~\frac{t}{N})}(\theta_Q)}\bigg|<
A^{\frac{1}{2}\big(\mathbf{B}_2(0)-\mathbf{B}_2(\frac{s}{N})\big)}
\bigg|\frac{1-\zeta_N^w}{1-A^\frac{s}{N}}\bigg|\prod_{n=1}^\infty\frac{(1+A^n)^2}{{(1-A^{n+\frac{s}{N}})}
{(1-A^{n-\frac{s}{N}})}}\quad\textrm{by (\ref{FourierSiegel})}\\
&<&
A^{\frac{1}{2}\big(\mathbf{B}_2(0)-\mathbf{B}_2(\frac{1}{N})\big)}
\frac{2}{1-A^\frac{1}{N}}\prod_{n=1}^\infty\frac{(1+A^n)^2}{{(1-A^{n})}
{(1-A^{n-\frac{1}{2}})}}\quad\textrm{by}~1\leq s\leq\tfrac{N}{2}\\
&<&
\frac{2A^{\frac{1}{2}(\frac{1}{N}-\frac{1}{N^2})}}{1-A^\frac{1}{N}}\prod_{n=1}^\infty
(1+A^n)^2(1+A^\frac{n}{1.03})(1+A^{\frac{1}{1.03}(n-\frac{1}{2})})\quad\textrm{by
Lemma \ref{ineq}(iii)}\\
&\stackrel{}{<}& \frac{2A^{\frac{1}{4N}}}{1-A^\frac{1}{N}}
\prod_{n=1}^\infty
e^{2A^{n}+A^\frac{n}{1.03}+A^{\frac{1}{1.03}(n-\frac{1}{2})}}\quad\textrm{by
the fact $N\geq2$ and Lemma \ref{ineq}(iv)}\\
&=&\frac{2A^{\frac{1}{4N}}}{1-A^\frac{1}{N}} e^{\frac{2A}{1-A}+
\frac{A^{\frac{1}{1.03}}+A^{\frac{1}{2.06}}}{1-A^{\frac{1}{1.03}}}}\leq
\frac{2e^{-\frac{\pi\sqrt{-d_K}}{4N}}}{1-e^\frac{-\pi\sqrt{-d_K}}{N}}
e^{\frac{2e^{-\sqrt{43}\pi}}{1-e^{-\sqrt{43}\pi}}+
\frac{e^{-\frac{\sqrt{43}\pi}{1.03}}+e^{-\frac{\sqrt{43}\pi}{2.06}}}{1-e^{-\frac{\sqrt{43}\pi}{1.03}}}}
\quad\textrm{by the fact $d_K\leq-43$}\\
&<&\frac{2.0001e^{-\frac{\pi\sqrt{-d_K}}{4N}}}{1-e^\frac{-\pi\sqrt{-d_K}}{N}}
\leq\frac{2.0001e^{-\frac{\pi}{4}}}{1-e^{-\pi}}<1\quad\textrm{by the
fact $N\leq\sqrt{-d_K}$},
\end{eqnarray*}
which proves the lemma.
\end{proof}

\begin{remark}
Observe that the condition (\ref{N,d_K}) is stronger than
(\ref{condition2}), namely
\begin{equation*}
\frac{-\sqrt{3}\pi}{\ln\big(1-2.16e^{-\frac{\pi\sqrt{-d_K}}{24}}\big)}<\sqrt{-d_K}.
\end{equation*}
\end{remark}

Now we are ready to prove our main theorem about primitive
generators of ring class fields over $K$.

\begin{theorem}\label{main}
Assume the condition (\ref{N,d_K}) and let $\mathcal{O}$ be the
order of conductor $N$ in $K$.
 Then the singular value
\begin{equation}\label{value}
\prod_{\tiny\begin{smallmatrix}1\leq
w\leq\frac{N}{2}\\\gcd(w,~N)=1\end{smallmatrix}}
g^{\frac{12N}{\gcd(6,~N)}}_{(0,~\frac{w}{N})}(\theta)
\end{equation}
generates $H_\mathcal{O}$ over $K$. It is a real algebraic integer
and its minimal polynomial has integer coefficients. In particular,
if the conductor $N$ has at least two prime factors, then it is a
unit.
\end{theorem}
\begin{proof}
Let $g(\tau)=\prod_{\tiny\begin{smallmatrix}1\leq
w\leq\frac{N}{2}\\\gcd(w,~N)=1\end{smallmatrix}}
g^{\frac{12N}{\gcd(6,~N)}}_{(0,~\frac{w}{N})}(\tau)$. By
(\ref{K_N/H}) and Theorem \ref{H_O/H} we have
$\mathrm{Gal}(K_{(N)}/H_\mathcal{O})\cong
\big\{\left(\begin{smallmatrix}t&0\\0&t\end{smallmatrix}\right)~:~t\in(\mathbb{Z}/N\mathbb{Z})^*\big\}/
\big\{
\pm\left(\begin{smallmatrix}1&0\\0&1\end{smallmatrix}\right)\big\}$,
and hence
\begin{equation*}
g(\theta)=\prod_w
\bigg(g^{\frac{12N}{\gcd(6,~N)}}_{(0,~\frac{1}{N})}(\tau)\bigg)^{\left(\begin{smallmatrix}w&0\\0&w\end{smallmatrix}\right)}(\theta)
=\prod_w
\bigg(g^{\frac{12N}{\gcd(6,~N)}}_{(0,~\frac{1}{N})}(\theta)\bigg)^{\left(\begin{smallmatrix}w&0\\0&w\end{smallmatrix}\right)}
=\mathbf{N}_{K_{(N)}/H_\mathcal{O}}\bigg(g_{(0,~\frac{1}{N})}^\frac{12N}{\gcd(6,~N)}(\theta)\bigg)
\end{equation*}
by Corollary \ref{F_N} and (\ref{K_N/H}). Thus $g(\theta)$ belongs
to $H_\mathcal{O}$. Now, if we show that the element of
$\mathrm{Gal}(H_\mathcal{O}/K)$ fixing the value $g(\theta)$ is only
the identity, then we can conclude by Galois theory that it
generates $H_\mathcal{O}$ over $K$.
\par
It follows from Theorem
\ref{conjugate} that any conjugate of $g(\theta)$ is of the form
\begin{equation*}
g^{\gamma\cdot \beta_Q}(\theta_Q)
\end{equation*}
for some $\gamma=\left(\begin{smallmatrix}t-B_\theta s&-C_\theta
s\\s&t\end{smallmatrix}\right)\in W_{N,~\theta}$ and
$Q=aX^2+bXY+cY^2\in\mathrm{C}(d_K)$. Assuming
$g(\theta)=g^{\gamma\cdot \beta_Q}(\theta_Q)$ we derive
\begin{equation*}
\prod_w g^{\frac{12N}{\gcd(6,~N)}}_{(0,~\frac{w}{N})}(\theta)
=\prod_w g^{\frac{12N}{\gcd(6,~N)}}_{(0,~\frac{w}{N})\gamma
\beta_Q}(\theta_Q)
\end{equation*}
by Corollary \ref{F_N}. Since $|g(\theta)|=|g^{\gamma\cdot
\beta_Q}(\theta_Q)|$, Lemma \ref{newlemma1} leads us to have $a=1$.
This yields
\begin{eqnarray*}
Q=\mathrm{id}=\left\{\begin{array}{ll} X^2-\frac{d_K}{4}Y^2 &
\textrm{for}~d_K\equiv0\pmod{4}\\
X^2+XY+\frac{1-d_K}{4}Y^2 & \textrm{for}~
d_K\equiv1\pmod{4}\end{array}\right.
\end{eqnarray*}
from the condition (\ref{reduced}) and the relation (\ref{disc});
hence
$\beta_Q=\left(\begin{smallmatrix}1&0\\0&1\end{smallmatrix}\right)$
as an element of $\mathrm{GL}_2(\mathbb{Z}/N\mathbb{Z})$ by
definitions (\ref{u1}) and (\ref{u2}), and $\theta_Q=\theta$ by
definition (\ref{theta_Q}). And we see from Corollary \ref{F_N} that
\begin{eqnarray*}
g(\theta)=g^{\gamma\cdot \beta_Q}(\theta_Q)=\prod_w
g_{(0,~\frac{w}{N})\gamma \beta_Q}^\frac{12N}{\gcd(6,~N)}(\theta_Q)=
\prod_w
g_{(\frac{ws}{N},~\frac{wt}{N})}^\frac{12N}{\gcd(6,~N)}(\theta)
\end{eqnarray*}
from which we get $s\equiv0\pmod{N}$ by Lemma \ref{newlemma2}.
Therefore the pair of
$\gamma=\left(\begin{smallmatrix}t&0\\0&t\end{smallmatrix}\right)$
and $Q=\mathrm{id}$ represents the identity on $H_\mathcal{O}$ (see
the tower in the proof of Theorem \ref{conjugate}), and hence
$g(\theta)$ actually generates $H_\mathcal{O}$ over $K$.
\par
On the other hand, we derive from the definition
(\ref{FourierSiegel})
\begin{eqnarray*}
g(\theta)&=&\prod_w
\bigg\{q_\theta^{\frac{1}{12}}(1-\zeta_N^w)\prod_{n=1}^\infty(1-q_\theta^n\zeta_N^w)(1-q_\theta^n\zeta_N^{-w})\bigg\}^\frac{12N}
{\gcd(6,~N)}\\
&=&\prod_w\bigg\{ q_\theta^\frac{N}{\gcd(6,~N)}
\big(2\sin\tfrac{w\pi}{N}\big)^\frac{12N}{\gcd(6,~N)}\prod_{n=1}^\infty\big(1-2\cos\tfrac{2w\pi}{N}q_\theta^n+q_\theta^{2n}\big)^\frac{12N}{\gcd(6,~N)}\bigg\},
\end{eqnarray*}
and this claims that $g(\theta)$ is a real number. Furthermore, we
see from Proposition \ref{transformation}(i) that the function
$g(\tau)$ is integral over $\mathbb{Z}[j(\tau)]$. Since $j(\theta)$
is a real algebraic integer (\cite{Lang} or \cite{Shimura}), so is
the value $g(\theta)$. And its minimal polynomial over $K$ has
integer coefficients. In particular, if $N$ has at least two prime
factors, the function $1/g(\tau)$ is also integral over
$\mathbb{Z}[j(\tau)]$ by Proposition \ref{transformation}(ii); hence
$g(\theta)$ becomes a unit.
\end{proof}

\begin{remark}
Since the proof of Theorem \ref{main} depends only on Lemmas
\ref{newlemma1} and \ref{newlemma2} which do not include any power
of singular values, any nonzero power of the value in (\ref{value})
can be also a generator of $H_\mathcal{O}$ over $K$.
\end{remark}

\begin{remark}\label{example}
We would like to present an example which cannot be covered by our
method due to violation of the condition (\ref{N,d_K}).
\par
Let $K=\mathbb{Q}(\sqrt{-5})$ and $N=12(=2^2\cdot3)$. Then
$d_K=-20$, $\theta=\sqrt{-5}$ and
\begin{eqnarray*}
&&\mathrm{C}(d_K)=\big\{Q_1=X^2+5Y^2,\quad Q_2=2X^2+2XY+3Y^2\big\}\\
&&\theta_{Q_1}=\sqrt{-5},\quad\theta_{Q_2}=\tfrac{-1+\sqrt{-5}}{2}\\
&&\beta_{Q_1}=\left(\begin{smallmatrix}
1&0\\0&1\end{smallmatrix}\right),
\quad\beta_{Q_2}=\left(\begin{smallmatrix}1&5\\3&2\end{smallmatrix}\right)\\
&&W_{N,~\theta}/\big\{\left(\begin{smallmatrix}t&0\\0&t\end{smallmatrix}\right)~:~t\in(\mathbb{Z}/N\mathbb{Z})^*\big\}=
\big\{ \left(\begin{smallmatrix}1&0\\0&1\end{smallmatrix}\right),~
\left(\begin{smallmatrix}1&6\\6&1\end{smallmatrix}\right),~
\left(\begin{smallmatrix}2&9\\3&2\end{smallmatrix}\right),~
\left(\begin{smallmatrix}3&2\\2&3\end{smallmatrix}\right),~
\left(\begin{smallmatrix}3&4\\4&3\end{smallmatrix}\right),~
\left(\begin{smallmatrix}4&9\\3&4\end{smallmatrix}\right),~
\left(\begin{smallmatrix}6&7\\1&6\end{smallmatrix}\right),~
\left(\begin{smallmatrix}0&7\\1&0\end{smallmatrix}\right)\big\}.
\end{eqnarray*}
Now, the conjugates of
\begin{equation*}
x=\prod_{\tiny\begin{smallmatrix}1\leq
w\leq\frac{N}{2}\\\gcd(w,~N)=1\end{smallmatrix}}
g_{(0,~\frac{w}{N})}^\frac{12N}{\gcd(6,~N)}(\theta)=
g_{(0,~\frac{1}{12})}^{24}(\sqrt{-5})g_{(0,~\frac{5}{12})}^{24}(\sqrt{-5})
\end{equation*}
are as follows:
\begin{eqnarray*}
\begin{array}{llll}
x_{1\phantom{1}}=g_{(0,~\frac{1}{12})}^{24}(\sqrt{-5})g_{(0,~\frac{5}{12})}^{24}(\sqrt{-5}),
&
x_{2\phantom{1}}=g_{(\frac{6}{12},~\frac{1}{12})}^{24}(\sqrt{-5})g_{(\frac{6}{12},~\frac{5}{12})}^{24}(\sqrt{-5})\\
x_{3\phantom{1}}=g_{(\frac{3}{12},~\frac{2}{12})}^{24}(\sqrt{-5})g_{(\frac{3}{12},~\frac{10}{12})}^{24}(\sqrt{-5}),
&
x_{4\phantom{1}}=g_{(\frac{2}{12},~\frac{3}{12})}^{24}(\sqrt{-5})g_{(\frac{10}{12},~\frac{3}{12})}^{24}(\sqrt{-5})\\
x_{5\phantom{1}}=g_{(\frac{4}{12},~\frac{3}{12})}^{24}(\sqrt{-5})g_{(\frac{8}{12},~\frac{3}{12})}^{24}(\sqrt{-5}),
&
x_{6\phantom{1}}=g_{(\frac{3}{12},~\frac{4}{12})}^{24}(\sqrt{-5})g_{(\frac{3}{12},~\frac{8}{12})}^{24}(\sqrt{-5})\\
x_{7\phantom{1}}=g_{(\frac{1}{12},~\frac{6}{12})}^{24}(\sqrt{-5})g_{(\frac{5}{12},~\frac{6}{12})}^{24}(\sqrt{-5}),
&
x_{8\phantom{1}}=g_{(\frac{1}{12},~0)}^{24}(\sqrt{-5})g_{(\frac{5}{12},~0)}^{24}(\sqrt{-5})\\
x_{9\phantom{1}}=g_{(\frac{3}{12},~\frac{2}{12})}^{24}(\tfrac{-1+\sqrt{-5}}{2})g_{(\frac{3}{12},~\frac{10}{12})}^{24}(\tfrac{-1+\sqrt{-5}}{2}),
&
x_{10}=g_{(\frac{9}{12},~\frac{8}{12})}^{24}(\tfrac{-1+\sqrt{-5}}{2})g_{(\frac{9}{12},~\frac{4}{12})}^{24}(\tfrac{-1+\sqrt{-5}}{2})\\
x_{11}=g_{(\frac{9}{12},~\frac{7}{12})}^{24}(\tfrac{-1+\sqrt{-5}}{2})g_{(\frac{9}{12},~\frac{11}{12})}^{24}(\tfrac{-1+\sqrt{-5}}{2}),
&
x_{12}=g_{(\frac{11}{12},~\frac{4}{12})}^{24}(\tfrac{-1+\sqrt{-5}}{2})g_{(\frac{7}{12},~\frac{8}{12})}^{24}(\tfrac{-1+\sqrt{-5}}{2})\\
x_{13}=g_{(\frac{1}{12},~\frac{2}{12})}^{24}(\tfrac{-1+\sqrt{-5}}{2})g_{(\frac{5}{12},~\frac{10}{12})}^{24}(\tfrac{-1+\sqrt{-5}}{2}),
&
x_{14}=g_{(\frac{3}{12},~\frac{11}{12})}^{24}(\tfrac{-1+\sqrt{-5}}{2})g_{(\frac{3}{12},~\frac{7}{12})}^{24}(\tfrac{-1+\sqrt{-5}}{2})\\
x_{15}=g_{(\frac{7}{12},~\frac{5}{12})}^{24}(\tfrac{-1+\sqrt{-5}}{2})g_{(\frac{11}{12},~\frac{1}{12})}^{24}(\tfrac{-1+\sqrt{-5}}{2}),
&
x_{16}=g_{(\frac{1}{12},~\frac{5}{12})}^{24}(\tfrac{-1+\sqrt{-5}}{2})g_{(\frac{5}{12},~\frac{1}{12})}^{24}(\tfrac{-1+\sqrt{-5}}{2})
\end{array}
\end{eqnarray*}
possibly with multiplicity by Theorem \ref{conjugate} and Corollary
\ref{F_N}. Hence the minimal polynomial of $x$ over $K$ would be
{\small\begin{eqnarray*}
&&\quad(X-x_1)\quad\cdots\quad(X-x_{16})\\
&&=X^{16}-1597283771136X^{15}+218685334974106886200X^{14}-989798760399582851353280X^{13}\\
&&+1635793922011311753339695900X^{12}-1478170408753689677872738383488X^{11}\\
&&+813690304957218006590231416378248X^{10}-464728779160514526974626326247201600X^9\\
&&+167117715935951295057696524156063178310X^8-9155763998650223557795196487031471321600X^7\\
&&-17410059883612682120508988571419246981752X^6-31984181681760551803330979365226550023488X^5\\
&&+5677583625730635496464554293769775900X^4-2249102100642965467076167124913280X^3\\
&&+238110589893565910129238086200X^2-2550974942476760820051136X+1.
\end{eqnarray*}}
And this polynomial is irreducible over $K$, so $x$ is indeed a
primitive generator of the ring class field of the order of
conductor $12$ in $\mathbb{Q}(\sqrt{-5})$. Moreover, $x$ is a unit
because the constant term is $1$. Therefore, it would be worthwhile
to check how much further one can release from the condition
(\ref{N,d_K}). On the other hand, in the next section we will find
in a different way a ring class invariant as singular value of
certain quotient of the $\Delta$-function without the condition
(\ref{N,d_K}) when the conductor of the extension $H_\mathcal{O}/K$
is a prime power.
\end{remark}

From now on, we will investigate how the ring class invariant
(\ref{value}) is related to a normal basis of $H_\mathcal{O}$ over
$K$. Even though the following four lemmas were studied in
\cite{J-K-S}, we will present their proofs for the sake of
completeness. Let $F$ be a finite abelian extension of a number
field $L$ with
$G=\textrm{Gal}(F/L)=\{\gamma_1=\textrm{id},\cdots,\gamma_n\}$.

\begin{lemma}\label{det}
A set of elements $\{x_1,\cdots,x_n\}$ in $F$ is a $L$-basis of $F$
if and only if
\begin{equation*}
\det\big(x_k^{\gamma_\ell^{-1}}\big)_{1\leq k,~\ell\leq n}\neq0.
\end{equation*}
\end{lemma}
\begin{proof}
Straightforward.
\end{proof}

By $\widehat{G}$ we denote the character group of $G$. Then we have
the \textit{Frobenius determinant relation}:

\begin{lemma}\label{Frobenius}
If $f$ is any $\mathbb{C}$-valued function on $G$, then
\begin{equation*}
\prod_{\chi\in\widehat{G}}\sum_{1\leq k\leq
n\phantom{\widehat{`}}}{\chi}(\gamma_k^{-1})f(\gamma_k)=
\det\big(f(\gamma_k \gamma_\ell^{-1})\big)_{1\leq k,~\ell\leq n}.
\end{equation*}
\end{lemma}
\begin{proof}
See \cite{Lang} Chapter 21 Theorem 5.
\end{proof}

Combining Lemma \ref{det} and Lemma \ref{Frobenius} we derive the
following lemma:

\begin{lemma}\label{character}
The conjugates of an element $x\in F$ form a normal basis of $F$
over $L$ if and only if
\begin{equation*}
\sum_{1\leq k\leq n }{\chi}(\gamma_k^{-1})x^{\gamma_k}\neq0
\quad\textrm{for all $\chi\in\widehat{G}$}.
\end{equation*}
\end{lemma}
\begin{proof}
For an element $x\in F$, set $x_k=x^{\gamma_k}$ for $1\leq k\leq n$.
Then we get that
\begin{eqnarray*}
&&\textrm{the conjugates of $x$ form a normal basis of $F$ over
$L$}\\
&\Longleftrightarrow&\textrm{\{$x_1,\cdots,x_n\}$ is a $L$-basis of
$F$}~\textrm{by the definition of a normal basis}\\
&\Longleftrightarrow& \det\big(x_k^{\gamma_\ell^{-1}}\big)_{1\leq
k,~\ell\leq
n}\neq0\quad\textrm{by Lemma \ref{det}}\\
&\Longleftrightarrow& \sum_{1\leq k\leq
n}{\chi}(\gamma_k^{-1})x_k\neq0\quad\textrm{for all
$\chi\in\widehat{G}$ by Lemma \ref{Frobenius} with}~f(\gamma_k)=x_k.
\end{eqnarray*}
\end{proof}

Now we present a simple criterion which enables us to determine
whether the conjugates of an element $x\in F$ form a normal basis of
$F$ over $L$.

\begin{lemma}\label{criterion}
Assume that there exists an element $x\in F$ such that
\begin{equation}\label{smaller}
\bigg|\frac{x^{\gamma_k}}{x}\bigg|<1\quad\textrm{for $1<k\leq n$}.
\end{equation} Then the conjugates of a high power
of $x$ form a normal basis of $F$ over $L$.
\end{lemma}
\begin{proof}
By the hypothesis (\ref{smaller}) we can take a suitably large
integer $m$ such that
\begin{equation}\label{smaller2}
\bigg|\frac{x^{\gamma_k}}{x}\bigg|^m\leq\frac{1}{\#G}\quad\textrm{for
$1<k\leq n$}
\end{equation}
where $\#G$ is the cardinality of $G$. Then for $\chi\in\widehat{G}$
we have
\begin{eqnarray*}
\bigg|\sum_{1\leq k\leq
n}{\chi}(\gamma_k^{-1})(x^m)^{\gamma_k}\bigg| &\geq&
|x^m|\bigg(1-\sum_{1<k\leq
n}\bigg|\frac{(x^m)^{\gamma_k}}{x^m}\bigg|\bigg)\quad\textrm{by the
triangle inequality}\\
&\geq&|x^m|\bigg(1-\frac{1}{\#G}(\#G-1)\bigg)=\frac{|x^m|}{\#G}>0\quad\textrm{by
(\ref{smaller2}).}
\end{eqnarray*}
Therefore the conjugates of $x^m$ form a normal basis of $F$ over
$L$ by Lemma \ref{character}.
\end{proof}

\begin{theorem}\label{normal1}
Assume the condition (\ref{N,d_K}) and let $\mathcal{O}$ be the
order of conductor $N$ in $K$. Then the conjugates of a high power
of
\begin{equation}\label{value1}
\prod_{\tiny\begin{smallmatrix}1\leq
w\leq\frac{N}{2}\\\gcd(w,~N)=1\end{smallmatrix}}
g^{-\frac{12N}{\gcd(6,~N)}}_{(0,~\frac{w}{N})}(\theta)
\end{equation}
form a normal basis of $H_\mathcal{O}$ over $K$.
\end{theorem}
\begin{proof}
Let $x$ be the value in (\ref{value1}). We then see from the proof
of Theorem \ref{main} that $|x^\gamma/x|<1$ for all
$\gamma\neq\mathrm{id}\in\mathrm{Gal}(H_\mathcal{O}/K)$. Therefore,
the result follows from Lemma \ref{criterion}.
\end{proof}

\section{Generators of class fields with conductors of prime
power}\label{section5}

Let
\begin{equation}\label{Delta}
\Delta(\tau)=(2\pi
i)^{12}q_\tau\prod_{n=1}^\infty(1-q_\tau^n)^{24}\quad(\tau\in\mathfrak{H})
\end{equation}
be the \textit{$\Delta$-function} (or, discriminant function). In
this section we shall construct primitive generators of ring class
fields with conductor of prime power by utilizing singular values of
the $\Delta$-function.
\par
Throughout this section we let $K$ be an imaginary quadratic field
with discriminant $d_K$ and $\mathcal{O}_K=[\theta,~1]$ be its ring
of integers with $\theta\in\mathfrak{H}$. For a nonzero integral
ideal $\mathfrak{f}$ of $K$ we denote by $\mathrm{Cl}(\mathfrak{f})$
the ray class group of conductor $\mathfrak{f}$ and write $C_0$ for
its unit class. If $\mathfrak{f}\neq\mathcal{O}_K$ and
$C\in\mathrm{Cl}(\mathfrak{f})$, then we take an integral ideal
$\mathfrak{c}$ in $C$ so that
$\mathfrak{f}\mathfrak{c}^{-1}=[z_1,~z_2]$ with
$z=z_1/z_2\in\mathfrak{H}$. Now we define the
\textit{Siegel-Ramachandra invariant} by
\begin{equation*}
g_\mathfrak{f}(C)=g_{(\frac{a}{N},~\frac{b}{N})}^{12N}(z)
\end{equation*}
where $N$ is the smallest positive integer in $\mathfrak{f}$ and
$a,~b\in\mathbb{Z}$ such that $1=\frac{a}{N}z_1+\frac{b}{N}z_2$.
This value depends only on the class $C$ and belongs to the ray
class field $K_\mathfrak{f}$ modulo $\mathfrak{f}$ of $K$.
Furthermore, we have a well-known transformation formula
\begin{equation}\label{Artin}
g_\mathfrak{f}(C_1)^{\sigma(C_2)}=g_\mathfrak{f}(C_1C_2)
\end{equation}
for $C_1,~C_2\in\mathrm{Cl}(\mathfrak{f})$ where $\sigma$ is the
Artin map (\cite{K-L} Chapter 11 Section 1).
\par
Let $\chi$ be a character of $\mathrm{Cl}(\mathfrak{f})$. We then
denote by $\mathfrak{f}_\chi$ the conductor of $\chi$ and let
$\chi_0$ be the proper character of $\mathrm{Cl}(\mathfrak{f}_\chi)$
corresponding to $\chi$. For a nontrivial character $\chi$ of
$\mathrm{Cl}(\frak{f})$ with $\mathfrak{f}\neq\mathcal{O}_K$ we
define
\begin{eqnarray*}
S_\mathfrak{f}(\chi,~g_\mathfrak{f})
&=&\sum_{C\in\mathrm{Cl}(\mathfrak{f})}
\chi(C)\log|g_\mathfrak{f}(C)|\\
L_\mathfrak{f}(s,~\chi)&=&\sum_{\begin{smallmatrix}\mathfrak{a}\neq0~:~\textrm{integral
ideals}\\\gcd(\mathfrak{a},~\mathfrak{f})=\mathcal{O}_K\end{smallmatrix}}\frac{\chi(\mathfrak{a})}{\mathbf{N}_{K/\mathbb{Q}}(\mathfrak{a})^s}\quad(s\in\mathbb{C}).
\end{eqnarray*}
If $\mathfrak{f}_\chi\neq\mathcal{O}_K$, then we see from the second
Kronecker limit formula that
\begin{equation*}
L_{\mathfrak{f}_\chi}(1,~\chi_0)=T_0S_{\mathfrak{f}_\chi}(\overline{\chi}_0,~g_{\mathfrak{f}_\chi})
\end{equation*}
where $T_0$ is a nonzero constant depending on $\chi_0$ (\cite{Lang}
Chapter 22 Theorem 2). Here we observe that the value
$L_{\mathfrak{f}_\chi}(1,~\chi_0)$ is nonzero (\cite{Janusz} Chapter
IV Proposition 5.7). Moreover, multiplying the above relation by the
Euler factors we derive the identity
\begin{equation}\label{relation}
\prod_{\mathfrak{p}|\mathfrak{f},~
\mathfrak{p}\nmid\mathfrak{f}_\chi}\big(1-\overline{\chi}_0(\mathfrak{p})\big)L_{\mathfrak{f}_\chi}(1,~\chi_0)
=TS_\mathfrak{f}(\overline{\chi},~g_\mathfrak{f})
\end{equation}
where $T$ is a nonzero constant depending on $\mathfrak{f}$ and
$\chi$ (\cite{K-L} p. 244).

\begin{theorem}\label{generator}
Let $L$ be an abelian extension of $K$ with $[L:K]>2h_K$ where $h_K$
is the class number of $K$. Assume that the conductor of the
extension $L/K$ is a power of prime ideal, namely
$\mathfrak{f}=\mathfrak{p}^n$ ($n\geq1$). Then the value
\begin{equation*}
\varepsilon=\mathbf{N}_{K_\mathfrak{f}/L}\bigg(g_\mathfrak{f}(C_0)\bigg)
\end{equation*}
generates $L$ over $K$.
\end{theorem}
\begin{proof}
We identify $\mathrm{Gal}(K_\mathfrak{f}/K)$ with
$\mathrm{Cl}(\mathfrak{f})$ via the Artin map. Letting
$F=K(\varepsilon)$ we deduce
\begin{equation}\label{number1}
\#\big\{\textrm{characters}~\chi~\mathrm{of}~\mathrm{Cl}(\mathfrak{f})~:~
\chi|_{\mathrm{Gal}(K_\mathfrak{f}/L)}=1~\textrm{and}~
\chi|_{\mathrm{Gal}(K_\mathfrak{f}/F)}\neq1\big\}=[L:K]-[F:K].
\end{equation}
Furthermore, if we let $H$ be the Hilbert class field of $K$, then
we have
\begin{equation}\label{number2}
\#\big\{\textrm{characters}~\chi~\mathrm{of}~\mathrm{Cl}(\mathfrak{f})~:~\mathfrak{f}_\chi=\mathcal{O}_K\big\}=
\#\big\{\chi~:~\chi|_{\mathrm{Gal}(K_\mathfrak{f}/H)}=1\big\}=h_K.
\end{equation}
Suppose that $F$ is properly contained in $L$. Then we deduce
\begin{equation*}
[L:K]-[F:K]=[L:K]\bigg(1-\frac{1}{[L:F]}\bigg)>2h_K\bigg(1-\frac{1}{2}\bigg)=h_K
\end{equation*}
by the hypothesis $[L:K]>2h_K$. Thus there exists a character $\psi$
of $\mathrm{Cl}(\mathfrak{f})$ such that
\begin{equation*}
\psi|_{\mathrm{Gal}(K_\mathfrak{f}/L)}=1,\quad
\psi|_{\mathrm{Gal}(K_\mathfrak{f}/F)}\neq1\quad\textrm{and}\quad
\mathfrak{f}_\psi\neq\mathcal{O}_K
\end{equation*}
by (\ref{number1}) and (\ref{number2}). Moreover, since
$\mathfrak{f}=\mathfrak{p}^n$, we get
$\mathfrak{f}_\psi=\mathfrak{p}^m$ for some $1\leq m\leq n$. Hence
we obtain by (\ref{relation}) that
\begin{equation*}
0\neq
L_{\mathfrak{f}_\psi}(1,~\psi_0)=TS_\mathfrak{f}(\overline{\psi},~g_\mathfrak{f})
\end{equation*}
for a nonzero constant $T$ and the proper character $\psi_0$ of
$\mathrm{Cl}(\mathfrak{f}_\psi)$ corresponding to $\psi$. On the
other hand, we get that
\begin{eqnarray*}
S_\mathfrak{f}(\overline{\psi},~g_\mathfrak{f})
&=&\sum_{C\in\mathrm{Cl}(\mathfrak{f})}\overline{\psi}(C)\log|g_\mathfrak{f}(C)|\\
&=&\sum_{\begin{smallmatrix}C_1\in\mathrm{Cl}(\mathfrak{f})\\C_1\hspace{-0.2cm}
\mod{\mathrm{Gal}(K_\mathfrak{f}/F)}\end{smallmatrix}}
\sum_{\begin{smallmatrix}C_2\in\mathrm{Gal}(K_\mathfrak{f}/F)\\
C_2\hspace{-0.2cm}\mod{\mathrm{Gal}(K_\mathfrak{f}/L)}\end{smallmatrix}}
\sum_{C_3\in\mathrm{Gal}(K_\mathfrak{f}/L)}
\overline{\psi}(C_1C_2C_3)\log|g_\mathfrak{f}(C_1C_2C_3)|\\
&=&\sum_{C_1}\overline{\psi}(C_1)\sum_{C_2}\overline{\psi}(C_2)\log|\varepsilon^{\sigma(C_1C_2)}|
\quad\textrm{by the fact}~\psi|_{\mathrm{Gal}(K_\mathfrak{f}/L)}=1~\textrm{and}~(\ref{Artin})\\
&=&\sum_{C_1}\overline{\psi}(C_1)\bigg(\sum_{C_2}\overline{\psi}(C_2)\bigg)\log|\varepsilon^{\sigma(C_1)}|\quad\textrm{by the fact}~\varepsilon\in F\\
&=&0\quad\textrm{by the
fact}~\psi|_{\mathrm{Gal}(K_\mathfrak{f}/F)}\neq1,
\end{eqnarray*}
which is a contradiction. Therefore $L=F$ as desired.
\end{proof}

\begin{remark}
Schertz achieved in \cite{Schertz2} a similar result for generators
of the ray class fields. However, there seems to be some defect in
his argument. For instance, in the proof of \cite{Schertz2} Lemma 1
he claimed that the conductor of a nontrivial character of
$\mathrm{Cl}(\mathfrak{p}^n)$ is nontrivial. But one can see that
his argument could be false if $h_K\geq2$ because in this case the
conductor of a character of $\mathrm{Cl}(\mathfrak{p}^n)$ induced
from one of $\mathrm{Cl}(\mathcal{O}_K)$ is obviously trivial.
\end{remark}

We apply this theorem to obtain ring class invariants in terms of
singular values of the $\Delta$-function. To this end we are in need
of certain relation between Siegel functions and the
$\Delta$-function.

\begin{lemma}\label{StoD}
Let $N\geq1$. Then we have the relation
\begin{equation*}
\prod_{w=1}^{N-1}g_{(0,~\frac{w}{N})}^{12}(\tau)=N^{12}\frac{\Delta(N\tau)}{\Delta(\tau)}
\end{equation*}
where the left hand side is understood to be $1$ when $N=1$.
\end{lemma}
\begin{proof}
Note the identity
\begin{equation}\label{identity}
\frac{1-X^N}{1-X}=1+X+\cdots+X^{N-1}=\prod_{w=1}^{N-1}(1-e^{\frac{2\pi
iw}{N}}X).
\end{equation}
We then derive for $N\geq2$ that
\begin{eqnarray*}
\prod_{w=1}^{N-1}g_{(0,~\frac{w}{N})}^{12}(\tau)&=&
\prod_{w=1}^{N-1}\bigg(q_\tau^\frac{1}{12} e^{-\frac{\pi
iw}{N}}(1-e^{\frac{2\pi
iw}{N}})\prod_{n=1}^\infty(1-q_\tau^ne^{\frac{2\pi
iw}{N}})(1-q_\tau^ne^{-\frac{2\pi iw}{N}})\bigg)^{12}\quad\textrm{by
(\ref{FourierSiegel})}\\
&=&q_\tau^{N-1}N^{12}\prod_{n=1}^\infty\bigg(\frac{1-q_\tau^{Nn}}{1-q_\tau^n}\bigg)^{24}
\quad\textrm{by the identity (\ref{identity})}\\
&=&N^{12}\frac{\Delta(N\tau)}{\Delta(\tau)}\quad\textrm{by
definition (\ref{Delta})}.
\end{eqnarray*}
\end{proof}

\begin{theorem}\label{maindelta}
Let $K$ be an imaginary quadratic field other than
$\mathbb{Q}(\sqrt{-1})$ and $\mathbb{Q}(\sqrt{-3})$. For a prime $p$
which is inert or ramified in $K/\mathbb{Q}$, let
$\mathcal{O}=[p^\ell\theta,~1]$ ($\ell\geq1$). Then the real
algebraic integer
\begin{equation}\label{Deltasingular}
p^{12}\frac{\Delta(p^\ell\theta)}{\Delta(p^{\ell-1}\theta)}
\end{equation}
generates $H_\mathcal{O}$ over $K$.
\end{theorem}
\begin{proof}
Let $\mathfrak{f}=p^\ell\mathcal{O}_K$. Then the conductor of the
extension $H_\mathcal{O}/K$ is $\mathfrak{f}$ (for instance, see
\cite{Cox} Exercises 9.20$\sim$9.23) and
\begin{equation*}
[H_\mathcal{O}:K]=\left\{\begin{array}{ll} p^{\ell-1}(p+1)h_K &
\textrm{if $p$ is inert in $K/\mathbb{Q}$}\\
p^\ell h_K & \textrm{if $p$ is ramified in $K/\mathbb{Q}$}
\end{array}\right.
\end{equation*}
by the class number formula (\cite{Cox} Theorem 7.24).
\par
If $p=2$ and $\ell=1$, then $K_\mathfrak{f}=H_\mathcal{O}$ by Remark
\ref{remark3.3} and hence the real algebraic integer
$g_{(0,~\frac{1}{2})}^{24}(\theta)$ generates $H_\mathcal{O}$ over
$K$. And,
$g_{(0,~\frac{1}{2})}^{24}(\theta)=\big(2^{12}\frac{\Delta(2\theta)}{\Delta(\theta)}\big)^2$
by Lemma \ref{StoD}.
\par
As for the other cases, since $\mathfrak{f}$ is a prime power and
$[H_\mathcal{O}:K]>2h_K$, the value
$\mathbf{N}_{K_\mathfrak{f}/H_\mathcal{O}}\big(g_\mathfrak{f}(C_0)\big)$
generates $H_\mathcal{O}$ over $K$ by Theorem \ref{generator}. And
we have
\begin{eqnarray*}
\mathbf{N}_{K_\mathfrak{f}/H_\mathcal{O}}\bigg(g_\mathfrak{f}(C_0)\bigg)^2&=&
\prod_{\begin{smallmatrix}1\leq w\leq
p^\ell-1\\\gcd(w,~p)=1\end{smallmatrix}}
g^{12p^\ell}_{(0,~\frac{w}{p^\ell})}(\theta)\quad\textrm{by (\ref{K_N/H}) and Theorem \ref{H_O/H}}\\
&=&\prod_{w=1}^{p^\ell-1}g_{(0,~\frac{w}{p^\ell})}^{12p^\ell}(\theta)\big/
\prod_{w=1}^{p^{\ell-1}-1}g_{(0,~\frac{pw}{p^\ell})}^{12p^\ell}(\theta)\\
&=&\bigg(p^{12\ell}\frac{\Delta(p^\ell\theta)}{\Delta(\theta)}\big/p^{12(\ell-1)}\frac{\Delta(p^{\ell-1}\theta)}{\Delta(\theta)}\bigg)^{p^\ell}\quad \textrm{by Lemma \ref{StoD}}\\
&=&\bigg(p^{12}\frac{\Delta(p^\ell\theta)}{\Delta(p^{\ell-1}\theta)}\bigg)^{p^\ell}.
\end{eqnarray*}
\par
On the other hand, since both
$\frac{\Delta(p^{\ell-1}\theta)}{\Delta(\theta)}$ and
$\frac{\Delta(p^{\ell}\theta)}{\Delta(\theta)}$ are real algebraic
numbers which belong to $H_\mathcal{O}$ (\cite{Lang} Chapter 12
Corollary to Theorem 1), we get the assertion.
\end{proof}

\begin{remark}
Unfortunately, however, we cannot guarantee the fact that the
conjugates of the value in (\ref{Deltasingular}) (or, its inverse)
constitute a normal basis of $H_\mathcal{O}$ over $K$.
\end{remark}

\section{Construction of normal bases}\label{section6}
Given an imaginary quadratic field
$K(\neq\mathbb{Q}(\sqrt{-1}),~\mathbb{Q}(\sqrt{-3}))$ we consider
the extension $K_{(p^2m)}/K_{(pm)}$ for a prime $p\geq5$ and an
integer $m\geq1$ relatively prime to $p$. In this section we shall
construct a normal basis of each intermediate field $F$ over
$K_{(pm)}$ in a different way from Section \ref{section4}, namely by
using the idea of Kawamoto (\cite{Kawamoto}).
\par
First we explicitly determine all intermediate fields $F$ between
$K_{(p^2m)}$ and $K_{(pm)}$. Let $\theta$ be as in (\ref{theta}) and
set $\min(\theta,~\mathbb{Q})=X^2+B_\theta X+C_\theta$. Then one can
identify $\Gamma=\mathrm{Gal}\big(K_{(p^2m)}/K_{(pm)}\big)$ with
\begin{equation*}
\bigg\{\gamma=\begin{pmatrix}t-B_\theta s & -C_\theta s\\s &
t\end{pmatrix}\in\mathrm{GL}_2(\mathbb{Z}/p^2m\mathbb{Z})~:~\gamma\equiv
\begin{pmatrix}1&0\\0&1\end{pmatrix}\pmod{pm}\bigg\}\big/\bigg\{
\pm\begin{pmatrix}1&0\\0&1\end{pmatrix}\bigg\}
\end{equation*}
by (\ref{K_N/H}). Since $[K_{(p^2m)}:K_{(pm)}]=p^2$ by the formula
(\ref{degree}), we readily know by inspection that
\begin{equation*}
\Gamma=\bigg\langle\begin{pmatrix} 1+pm & 0\\0 & 1+pm
\end{pmatrix}
\bigg\rangle\times \bigg\langle\begin{pmatrix} 1-B_\theta pm &
-C_\theta pm\\pm & 1
\end{pmatrix}
\bigg\rangle,
\end{equation*}
which shows that $\Gamma\cong(\mathbb{Z}/p\mathbb{Z})^2$. Hence an
element of $\Gamma$ is of the form
\begin{equation*}
\begin{pmatrix} 1+pm & 0\\0 & 1+pm
\end{pmatrix}^k
\begin{pmatrix} 1-B_\theta pm & -C_\theta pm\\pm & 1
\end{pmatrix}^\ell
=\begin{pmatrix} 1+(k-B_\theta \ell )pm & -C_\theta \ell pm\\\ell pm
& 1+kpm
\end{pmatrix}
\end{equation*}
for $0\leq k,~\ell \leq p-1$. Set
\begin{equation*}
\Gamma_{(k,~\ell
)}=\big\langle\gamma_{(k,~\ell)}\big\rangle=\bigg\langle
\begin{pmatrix} 1+(k-B_\theta \ell )pm & -C_\theta \ell
pm\\\ell pm & 1+kpm
\end{pmatrix}\bigg\rangle
\end{equation*}
for
$(k,~\ell)\in\big\{(0,~1),~(1,~0),~(1,~1),~\cdots,~(1,~p-1)\big\}$,
which represents all subgroups of $\Gamma$ of order $p$. And, let
$F_{(k,~\ell)}$ be its corresponding fixed field of
$\Gamma_{(k,~\ell)}$, namely
\begin{equation*}
F_{(k,~\ell)}=K_{(p^2m)}^{\Gamma_{(k,~\ell)}}\quad\textrm{for}~
(k,~\ell)\in\big\{(0,~1),~(1,~0),~(1,~1),~\cdots,~(1,~p-1)\big\}.
\end{equation*}
Then we have the field tower:
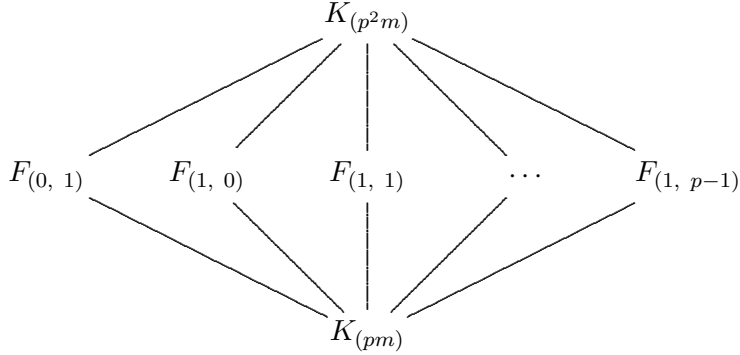
\begin{eqnarray*}
\begindc{\commdiag}
\obj(5,1)[A]{$K_{(pm)}$} \obj(1,3)[B]{$F_{(0,~1)}$}
\obj(3,3)[C]{$F_{(1,~0)}$} \obj(5,3)[D]{$F_{(1,~1)}$}
\obj(7,3)[E]{$\cdots$} \obj(9,3)[F]{$F_{(1,~p-1)}$}
\obj(5,5)[G]{$K_{(p^2m)}$} \mor{A}{B}{}[\atright,\solidline]
\mor{A}{C}{}[\atright,\solidline] \mor{A}{D}{}[\atright,\solidline]
\mor{A}{E}{}[\atright,\solidline] \mor{A}{F}{}[\atright,\solidline]
\mor{G}{B}{}[\atright,\solidline] \mor{G}{C}{}[\atright,\solidline]
\mor{G}{D}{}[\atright,\solidline] \mor{G}{E}{}[\atright,\solidline]
\mor{G}{F}{}[\atright,\solidline]
\enddc
\end{eqnarray*}

\begin{lemma}\label{pin}
$\zeta_{p},~g_{(0,~\frac{1}{pm})}^{12pm}(\theta)\in K_{(pm)}$ and
$\zeta_{p^2},~g_{(0,~\frac{1}{pm})}^{12m}(\theta)\in K_{(p^2m)}$.
\end{lemma}
\begin{proof}
One can check by Proposition \ref{modularity} that
$g_{(0,~\frac{1}{pm})}^{12pm}(\tau)\in\mathcal{F}_{pm}$ and
$g_{(0,~\frac{1}{pm})}^{12m}(\tau)\in\mathcal{F}_{p^2m}$. Hence we
get the assertion by (\ref{K_N}).
\end{proof}

Let us investigate the action of $\gamma_{(k,~\ell)}$ on
$\zeta_{p^2}$ and $g_{(0,~\frac{1}{pm})}^{12m}(\theta)$. To this end
we decompose $\gamma_{(k,~\ell)}$ into
\begin{eqnarray*}
\gamma_{(k,~\ell)}=
\alpha_{(k,~\ell)}\cdot\beta_{(k,~\ell)}&=&\begin{pmatrix}1&0\\0&1+(2k-B_\theta
\ell)pm\end{pmatrix}\begin{pmatrix}1+(k-B_\theta \ell)pm &
-C_\theta \ell pm\\\ell pm&1+(B_\theta \ell-k)pm\end{pmatrix}\\
&\in& \bigg\{\begin{pmatrix}1&0\\0&d\end{pmatrix}
~:~d\in(\mathbb{Z}/p^2m\mathbb{Z})^*\bigg\}\cdot
\mathrm{SL}_2(\mathbb{Z}/p^2m\mathbb{Z})\big/\bigg\{\pm
\begin{pmatrix}1&0\\0&1\end{pmatrix}\bigg\}.
\end{eqnarray*}
We see directly from (\ref{FourierSiegel}) that the function
$g_{(0,~\frac{1}{pm})}^{12m}(\tau)$ has Fourier coefficients in
$\mathbb{Q}(\zeta_{pm})$. Thus the action of $\alpha_{(k,~\ell)}$ is
described by (\ref{K_N/H}) and (\ref{first}) as
\begin{eqnarray*}
\zeta_{p^2}&\mapsto&\zeta_{p^2}^{1+(2k-B_\theta \ell)pm}\\
g_{(0,~\frac{1}{pm})}^{12m}(\theta)&\mapsto&
g_{(0,~\frac{1}{pm})}^{12m}(\theta).
\end{eqnarray*}
For some integers $A,~B,~C,~D$\quad let
\begin{equation*}
\beta_{(k,~\ell)}'=\begin{pmatrix} 1+(k-B_\theta \ell)pm+p^2mA &
-C_\theta \ell pm+p^2mB\\\ell pm+p^2mC & 1+(B_\theta \ell-k)pm+p^2mD
\end{pmatrix}
\end{equation*}
be a preimage of $\beta_{(k,~\ell)}$ via the natural surjection
$\mathrm{SL}_2(\mathbb{Z})\rightarrow\mathrm{SL}_2(\mathbb{Z}/p^2m\mathbb{Z})/\big\{\pm\left(\begin{smallmatrix}
1&0\\0&1\end{smallmatrix}\right)\big\}$. Then by (\ref{K_N/H}) and
(\ref{second}) we derive that the action of $\beta_{(k,~\ell)}$ is
given by
\begin{eqnarray*}
\zeta_{p^2}&\mapsto&\zeta_{p^2}\\
g_{(0,~\frac{1}{pm})}^{12m}(\theta)&\mapsto&
g_{(0,~\frac{1}{pm})}^{12m}(\tau)\circ\beta'_{(k,\ell)}(\theta)=
g_{(0,~\frac{1}{pm})\beta'_{(k,\ell)}}^{12m}(\theta)\quad\textrm{by Proposition \ref{transformation}(iii)}\\
&&=g_{(\ell+pC,~\frac{1}{pm}+B_\theta
\ell-k+pD)}^{12m}(\theta)=\zeta_{p^2}^{-6p\ell}g_{(0,~\frac{1}{pm})}^{12m}(\theta)\quad\textrm{by
Proposition \ref{transformation}(iv)}.
\end{eqnarray*}
Hence $\gamma_{(k,~\ell)}$ maps
\begin{eqnarray}\label{map1}
\zeta_{p^2} &\mapsto& \zeta_{p^2}^{1+(2k-B_\theta \ell)pm}\\
g_{(0,~\frac{1}{pm})}^{12m}(\theta) &\mapsto&
\zeta_{p^2}^{-6p\ell}g_{(0,~\frac{1}{pm})}^{12m}(\theta).\label{map2}
\end{eqnarray}

\begin{lemma}
Let
$(k,~\ell)\in\big\{(0,~1),~(1,~0),~(1,~1),~\cdots,~(1,~p-1)\big\}$.
Then $\Gamma_{(k,~\ell)}$ fixes $\zeta_{p^2}^x
g_{(0,~\frac{1}{pm})}^{12my}(\theta)$ for some $x,~y\in\mathbb{Z}$
if and only if $x$ and $y$ satisfy
\begin{equation}\label{cong}
\left\{\begin{array}{rrrll} -B_\theta mx & \equiv & 6y
&\hspace{-0.5cm} \pmod{p} &
\textrm{if}~(k,~\ell)=(0,~1)\\
x& \equiv &0 & \hspace{-0.5cm}\pmod{p} & \textrm{if}~(k,~\ell)=(1,~0)\\
(2-B_\theta \ell)mx & \equiv & 6\ell y & \hspace{-0.5cm}\pmod{p} &
\textrm{otherwise}.
\end{array}\right.
\end{equation}
\end{lemma}
\begin{proof}
It follows from (\ref{map1}) and (\ref{map2}) that
\begin{eqnarray*}
\bigg(\zeta_{p^2}^xg_{(0,~\frac{1}{pm})}^{12my}(\theta)\bigg)^{\gamma_{(k,~\ell)}}=\zeta_{p^2}^{\big(1+(2k-B_\theta
\ell)pm\big)x-6p\ell y} g_{(0,~\frac{1}{pm})}^{12my}(\theta).
\end{eqnarray*}
Then this value is equal to
$\zeta_{p^2}^xg_{(0,~\frac{1}{pm})}^{12my}(\theta)$ if and only if
\begin{equation*}
\big(1+(2k-B_\theta \ell)pm\big)x-6p\ell y\equiv x\pmod{p^2},
\end{equation*}
which reduces to (\ref{cong}). And, this proves the lemma.
\end{proof}

\begin{lemma}\label{composite}
$K_{(p^2m)}=K_{(pm)}\big(\zeta_{p^2},~g_{(0,~\frac{1}{pm})}^{12m}(\theta)\big)$.
\end{lemma}
\begin{proof}
Since $g_{(0,~\frac{1}{pm})}^{12m}(\theta)\not\in F_{(0,~1)}$ and
$g_{(0,~\frac{1}{pm})}^{12m}(\theta)\in F_{(1,~0)}$ by (\ref{map2}),
we claim that $g_{(0,~\frac{1}{pm})}^{12m}(\theta)\not\in K_{(pm)}$
and
$F_{(1,~0)}=K_{(pm)}\big(g_{(0,~\frac{1}{pm})}^{12m}(\theta)\big)$
owing to the fact $[F_{(1,~0)}:K_{(pm)}]=p$. Furthermore, since
$\zeta_{p^2}\not\in F_{(1,~0)}$ by (\ref{map1}), we achieve by the
fact $[K_{(p^2m)}:F_{(1,~0)}]=p$ that
\begin{equation*}
K_{(p^2m)}=F_{(1,~0)}\big(\zeta_{p^2}\big)=K_{(pm)}\bigg(\zeta_{p^2},~g_{(0,~\frac{1}{pm})}^{12m}(\theta)\bigg).
\end{equation*}
\end{proof}

\begin{theorem}\label{F}
Let
$(k,~\ell)\in\big\{(0,~1),~(1,~0),~(1,~1),~\cdots,~(1,~p-1)\big\}$
and $y'$ be the integer such that $y\cdot y'\equiv1\pmod{p}$ and
$0<y'<p$ for an integer $y\not\equiv0\pmod{p}$. Then we have
\begin{equation*}
F_{(k,~\ell)}=\left\{\begin{array}{ll}
K_{(pm)}\bigg(\zeta_{p^2}g_{(0,~\frac{1}{pm})}^{12m^2 6'(p-B_\theta)
}(\theta)\bigg) &
\textrm{if}~(k,~\ell)=(0,~1)\\
K_{(pm)}\bigg(g_{(0,~\frac{1}{pm})}^{12m}(\theta)\bigg) &
\textrm{if}~(k,~\ell)=(1,~0)\\
K_{(pm)}\bigg(\zeta_{p^2}g_{(0,~\frac{1}{pm})}^{12m^2(6\ell)'(2+p-B_\theta
\ell)}(\theta)\bigg) & \textrm{otherwise}.
\end{array}\right.
\end{equation*}
\end{theorem}
\begin{proof}
Take a solution of (\ref{cong}) as
\begin{eqnarray}\label{solution}
(x,~y)=\left\{\begin{array}{ll} \big(1,~m6'(p-B_\theta)\big) &
\textrm{if}~(k,~\ell)=(0,~1)\\
(0,~1) & \textrm{if}~(k,~\ell)=(1,~0)\\
\big(1,~m(6\ell)'(2+p-B_\theta \ell)\big) & \textrm{otherwise}
\end{array}\right.
\end{eqnarray}
which consists of nonnegative integers. We can then readily check
that a solution $(x,~y)$ in (\ref{solution}) does not satisfy two
congruence equations in (\ref{cong}) simultaneously. This shows that
for each $(x,~y)$,
$\zeta_{p^2}^xg_{(0,~\frac{1}{pm})}^{12my}(\theta)$ belongs to a
unique $F_{(k,~\ell)}$; hence in particular, it is not in
$K_{(pm)}$. Since $[F_{(k,~\ell)}:K_{(pm)}]=p$, we get the
conclusion.
\end{proof}

To accomplish our goal we are in need of the following two lemmas:

\begin{lemma}\label{lemma1}
Let $L$ be a number field containing $\zeta_n$ and $F$ be a cyclic
extension over $L$ of degree $n$. Then there exists an element $\xi$
of $L$ such that $F=L(\sqrt[n]{\xi})$. And, the conjugates of
$\sum_{s=0}^{n-1}(\sqrt[n]{\xi})^s$ over $L$ form a normal basis of
$F$ over $L$.
\end{lemma}
\begin{proof}
See \cite{Kawamoto} p. 223.
\end{proof}

\begin{lemma}\label{lemma2}
Let $L$ be a number field. Let $F_1$ and $F_2$ be finite Galois
extensions of $L$ with $F_1\cap F_2=L$. If the conjugates of
$\xi_s\in F_s$ over $L$ form a normal basis of $F_s$ over $L$ for
$s=1,~2$, then the conjugates of $\xi_1\xi_2$ over $L$ form a normal
basis of $F_1F_2$ over $L$.
\end{lemma}
\begin{proof}
See \cite{Kawamoto} p. 227.
\end{proof}

Now we are ready to prove our main theorem about normal bases.

\begin{theorem}\label{main1}
Let $(k,~\ell)$ and $y'$ be as in Theorem \ref{F}. Then the
conjugates of
\begin{equation}\label{values}
\left\{\begin{array}{ll}
\displaystyle\sum_{s=0}^{p-1}\bigg(\zeta_{p^2}g_{(0,~\frac{1}{pm})}^{12m^2
6'(p-B_\theta)}(\theta)\bigg)^s &
\textrm{if}~(k,~\ell)=(0,~1)\\
\displaystyle\sum_{s=0}^{p-1}g_{(0,~\frac{1}{pm})}^{12ms}(\theta) &
\textrm{if}~(k,~\ell)=(1,~0)\\
\displaystyle\sum_{s=0}^{p-1}\bigg(\zeta_{p^2}g_{(0,~\frac{1}{pm})}^{12m^2(6\ell)'(2+p-B_\theta
\ell)}(\theta)\bigg)^s & \textrm{otherwise}
\end{array}\right.
\end{equation}
over $K_{(pm)}$ form a normal basis of $F_{(k,~\ell)}$ over
$K_{(pm)}$. Moreover, the values in (\ref{values}) are algebraic
integers.
\end{theorem}
\begin{proof}
Since the function $g_{(0,~\frac{1}{pm})}(\tau)$ is integral over
$\mathbb{Z}[j(\tau)]$ by Proposition \ref{transformation}(i) and
$j(\theta)$ is an algebraic integer (\cite{Lang} or \cite{Shimura}),
the values in (\ref{values}) are all algebraic integers. Thus the
theorem follows by applying Lemma \ref{lemma1} with the aid of Lemma
\ref{pin} and Theorem \ref{F}.
\end{proof}

\begin{theorem}\label{main2}
The conjugates of the algebraic integer
\begin{eqnarray*}
\bigg(\sum_{s=0}^{p-1}\zeta_{p^2}^s\bigg)
\bigg(\sum_{s=0}^{p-1}g_{(0,~\frac{1}{pm})}^{12ms}(\theta)\bigg)
\end{eqnarray*}
over $K_{(pm)}$ form a normal basis of $K_{(p^2m)}$ over $K_{(pm)}$.
\end{theorem}
\begin{proof}
If $F_1=K_{(pm)}(\zeta_{p^2})$ and
$F_2=K_{(pm)}\big(g_{(0,~\frac{1}{pm})}^{12m}(\theta)\big)$, then
$[F_s:K_{(pm)}]\leq p$ for $s=1,~2$ by Lemma \ref{pin}. On the other
hand, Lemma \ref{composite} shows $F_1F_2=K_{(p^2m)}$, from which we
get $F_1\cap F_2=K_{(pm)}$ and $[F_s:K_{(pm)}]=p$ for $s=1,~2$.
Hence the conjugates of $\sum_{s=0}^{p-1}\zeta_{p^2}^s$ and
$\sum_{s=0}^{p-1}g_{(0,~\frac{1}{pm})}^{12ms}(\theta)$ over
$K_{(pm)}$ form normal bases of $F_1$ and $F_2$, respectively, by
Lemmas \ref{lemma1} and \ref{pin}. And, the theorem follows from
Lemma \ref{lemma2}.
\end{proof}

\section{Galois module structure}\label{section7}

Let $L$ be a number field and $p$ be an odd prime. We say that an
extension $L_\infty/L$ is a \textit{$\mathbb{Z}_p$-extension} of $L$
if there exists a sequence of cyclic extensions of $L$
\begin{equation*}
L=L_0\subset L_1\subset\cdots\subset L_n\subset\cdots\subset
L_\infty=\cup_{n=0}^\infty L_n
\end{equation*}
with $\mathrm{Gal}(L_n/L)\cong\mathbb{Z}/p^n\mathbb{Z}$. Then
$\mathrm{Gal}(L_\infty/L)\cong\mathbb{Z}_p$ and it is well-known
that $L_\infty/L$ is unramified outside $p$ (\cite{Washington}
Proposition 13.2). And, Greenberg (\cite{Greenberg}) has conjectured
that if $L$ is totally real, then the Iwasawa $\lambda$-invariant of
$L_\infty/L$ vanishes.
\par
Denoting the ring of $p$-integers of $L$ by
$\mathcal{O}_L[\frac{1}{p}]$ Kersten-Michali\v{c}ek introduced in
\cite{K-M} that a finite Galois extension $F$ of $L$ has a
\textit{normal $p$-integral basis over $L$} if
$\mathcal{O}_F[\frac{1}{p}]$ is a free
$\mathcal{O}_L[\frac{1}{p}]\mathrm{Gal}(F/L)$-module of rank one. We
then say that a $\mathbb{Z}_p$-extension $L_\infty$ of $L$ has a
\textit{normal basis over $L$} if each $L_n$ has a normal
$p$-integral basis over $L$.
\par
On the other hand, we see from \cite{Komatsu} that there is a
negative data for Greenberg's conjecture. For instance, for a
positive square free integer $d$ with $(\frac{-d}{3})=-1$, let
$L=\mathbb{Q}(\sqrt{3d})$ and $L'=\mathbb{Q}(\sqrt{-d})$. It was
shown in \cite{F-N} and \cite{F-K} that if $3$ divides the class
number of $L$ and if every $\mathbb{Z}_3$-extension of $L'$ has a
normal basis, then the $\lambda$-invariant of the cyclotomic
$\mathbb{Z}_3$-extension of $L$ does not vanish. This suggests a
relation between the existence of normal basis in
$\mathbb{Z}_p$-extension and the Greenberg's conjecture, which
motivates this section.
\par
Now, let $K(\neq\mathbb{Q}(\sqrt{-1}),~\mathbb{Q}(\sqrt{-3}))$ be an
imaginary quadratic field, $p\geq5$ be a prime and $m\geq1$ be an
integer relatively prime to $p$. And, let $n$ and $\ell$ be positive
integers with $n\geq2\ell$. Observe that the extension
$K_{(p^nm)}/K_{(p^\ell m)}$ is unramified outside $p$ (\cite{Cohen}
Chapter 3) and $\zeta_{p^n}\in K_{(p^nm)}$, $\zeta_{p^\ell}\in
K_{(p^\ell m)}$ but $\zeta_{p^{n+1}}\not\in K_{(p^nm)}$,
$\zeta_{p^{\ell+1}}\not\in K_{(p^\ell m)}$ (\cite{K-L} Chapter 9
Lemma 4.3). We shall prove in this section that $K_{(p^nm)}$ has a
normal $p$-integral basis over $K_{(p^\ell m)}$. The special case
for $\ell=1$ and $m=1$ has been done by Komatsu (\cite{Komatsu}).
However, we shall develop it in more comprehensive way by utilizing
(\ref{K_N/H}) and Proposition \ref{transformation} as in the
previous section unlike Komatsu's method via class field theory. As
a corollary we shall determine the existence of normal basis of the
$\mathbb{Z}_p$-extension $K_\infty K_{(p^\ell m)}/K_{(p^\ell m)}$
for $\ell\geq1$ where $K_\infty$ is any $\mathbb{Z}_p$-extension of
$K$.
\par
Let $\theta$ be as in (\ref{theta}) and set
$\min(\theta,~\mathbb{Q})=X^2+B_\theta X+C_\theta$. Then we can
identify the Galois group
$\Gamma=\mathrm{Gal}\big(K_{(p^{2(n-\ell)}m)}/K_{(p^\ell m)}\big)$
with the group
\begin{equation*}
\bigg\{ \gamma=\begin{pmatrix}t-B_\theta s &-C_\theta
s\\s&t\end{pmatrix}\in\mathrm{GL}_2(\mathbb{Z}/p^{2(n-\ell)}m\mathbb{Z})~:~
\gamma\equiv\begin{pmatrix}1&0\\0&1\end{pmatrix}\pmod{p^\ell
m}\bigg\}\big/ \bigg\{\pm\begin{pmatrix}1&0\\0&1\end{pmatrix}\bigg\}
\end{equation*}
by (\ref{K_N/H}) and $\#\Gamma=[K_{(p^{2(n-\ell)}m)}:K_{(p^\ell
m)}]=p^{2\big(2(n-\ell)-\ell\big)}$ by the formula (\ref{degree}).

\begin{lemma}\label{specific}
There exists an element $\beta_0$ of $\mathrm{SL}_2(\mathbb{Z})$
satisfying the property
\begin{equation}\label{property}
\beta_0^{p^{k}}=\begin{pmatrix}
* & * \\  p^{\ell+k}mq_k & *\end{pmatrix}
\equiv\begin{pmatrix}1&0\\0&1\end{pmatrix}\pmod{p^{\ell+k}m}
\qquad(0\leq k\leq2(n-\ell)-\ell)
\end{equation}
for some integers $q_k\not\equiv0\pmod{p}$.
\end{lemma}
\begin{proof}
Consider an integral matrix $\beta=\left(\begin{smallmatrix}1+p^\ell
mx-B_\theta p^\ell m & -C_\theta p^\ell m\\p^\ell m & 1+p^\ell
mx\end{smallmatrix}\right)$ for an undetermined integer $x$. Then
the condition $\det(\beta)\equiv1\pmod{p^{2(n-\ell)}m}$ is
equivalent to
\begin{equation}\label{ocong}
f(x)=p^\ell m^2x^2+(2m-B_\theta p^\ell m^2)x+C_\theta p^\ell m^2
-B_\theta m\equiv0\pmod{p^{2(n-\ell)-\ell}}.
\end{equation}
Since $2m-B_\theta p^\ell m^2\not\equiv0\pmod{p}$, the equation
$f(x)\equiv0\pmod{p}$ has a solution. Furthermore, since the
derivative $f'(x)=2m-B_\theta p^\ell m^2\not\equiv0\pmod{p}$, we
have an integer solution $x=x_0$ of the congruence equation
(\ref{ocong}) by Hensel's lemma. Let $\beta_0$ be a preimage of
$\left(\begin{smallmatrix}1+p^\ell mx_0-B_\theta p^\ell m &
-C_\theta p^\ell m\\p^\ell m&1+p^\ell mx_0\end{smallmatrix}\right)$
via the natural surjection
$\mathrm{SL}_2(\mathbb{Z})\rightarrow\mathrm{SL}_2(\mathbb{Z}/p^{2(n-\ell)}m\mathbb{Z})$.
Then it is routine to check that $\beta_0$ satisfies the property
(\ref{property}).
\end{proof}

Set $\alpha=\left(\begin{smallmatrix}1+p^\ell m & 0\\0&1+p^\ell
m\end{smallmatrix}\right)$ and $\beta=\beta_0$ in Lemma
\ref{specific}. Now that they have the order $p^{2(n-\ell)-\ell}$ in
$\Gamma$ and $\#\Gamma=p^{2\big(2(n-\ell)-\ell\big)}$, we derive
\begin{equation*}
\Gamma=\langle\alpha\rangle\times\langle\beta\rangle.
\end{equation*}
as a direct product. And, we get
\begin{equation}\label{galois}
\mathrm{Gal}\big(K_{(p^{2(n-\ell)}m)}/K_{(p^{n}m)}\big)=\langle\alpha^{p^{n-\ell}}\rangle\times\langle\beta^{p^{n-\ell}}\rangle
\end{equation}
by (\ref{K_N/H}). Let us define a function
\begin{equation*}
g(\tau)=\prod_{s=0}^{p^{n-2\ell}-1}g_{(0,~\frac{1}{p^{n-\ell}m})\beta^s}^{12m}(\tau).
\end{equation*}
Since each factor
$g_{(0,~\frac{1}{p^{n-\ell}m})\beta^s}^{12m}(\tau)$ lies in
$\mathcal{F}_{p^{2(n-\ell)}m}$ by Proposition \ref{modularity}, the
singular value $g(\theta)$ belongs to $K_{(p^{2(n-\ell)}m)}$ by
(\ref{K_N}).

\begin{lemma}\label{K_pn}
$K_{(p^nm)}=K_{(p^\ell m)}\big(\zeta_{p^n},~g(\theta)\big)$
\end{lemma}
\begin{proof}
By the property (\ref{property}) of $\beta$, $\beta^{p^{n-2\ell}}$
is of the form $\left(\begin{smallmatrix}1+p^{n-\ell}mA &
p^{n-\ell}mB\\p^{n-\ell}mC & 1+p^{n-\ell}mD\end{smallmatrix}\right)$
for some integers $A,~B,~C,~D$ with $C\not\equiv0\pmod{p}$. We then
deduce by (\ref{K_N/H}) and (\ref{second}) that
\begin{eqnarray}\label{beta}
g(\theta)^\beta&=&\prod_{s=0}^{p^{n-2\ell}-1}\bigg(g_{(0,~\frac{1}{p^{n-\ell}m})\beta^s}^{12m}(\theta)\bigg)^\beta
=\prod_{s=0}^{p^{n-2\ell}-1}g_{(0,~\frac{1}{p^{n-\ell}m})\beta^s\beta}^{12m}(\theta)\quad\textrm{by Proposition \ref{transformation}(iii)}\nonumber\\
&=&\frac{g^{12m}_{(0,~\frac{1}{p^{n-\ell}m})\beta^{p^{n-2\ell}}}(\theta)}{g^{12m}_{(0,~\frac{1}{p^{n-\ell}m})}(\theta)}~g(\theta)=
\frac{g^{12m}_{(C,~\frac{1}{p^{n-\ell}m}+D)}(\theta)}{g^{12m}_{(0,~\frac{1}{p^{n-\ell}m})}(\theta)}~g(\theta)\nonumber\\
&=&\zeta_{p^{n-\ell}}^{-6C}g(\theta)\quad\textrm{by Proposition
\ref{transformation}(iv)}.
\end{eqnarray}
In particular, $g(\theta)^{p^{n-\ell}}$ is fixed by $\beta$ and
$g(\theta)$ is fixed by $\beta^{p^{n-\ell}}$ because $\beta$ fixes
$\zeta_{p^{n-\ell}}$ by (\ref{K_N/H}) and (\ref{second}). Note that
$\alpha^{p^{n-2\ell}}=\left(\begin{smallmatrix}1+p^{n-\ell}mE &
0\\0&1+p^{n-\ell}mE\end{smallmatrix}\right)$ for some integer $E$.
As an element of $\Gamma$ we can decompose $\alpha^{p^{n-2\ell}}$
into
\begin{equation*}
\alpha^{p^{n-2\ell}}=\alpha_1\cdot\alpha_2=\begin{pmatrix}1&0\\0&(1+p^{n-\ell}mE)^2\end{pmatrix}
\begin{pmatrix}1+p^{n-\ell}mE+p^{2(n-\ell)}mA' &
p^{2(n-\ell)}mB'\\p^{2(n-\ell)}mC' &
E'+p^{2(n-\ell)}mD'\end{pmatrix}
\end{equation*}
for some integers $A',~B',~C',~D',~E'$ such that
$(1+p^{n-\ell}mE)E'\equiv1\pmod{p^{2(n-\ell)}m}$ and
$\alpha_2\in\mathrm{SL}_2(\mathbb{Z})$. Hence, we get by
(\ref{K_N/H}), (\ref{first}), (\ref{second}) and Proposition
\ref{transformation}(iii) that
\begin{eqnarray}\label{alpha}
g(\theta)^{\alpha^{p^{n-2\ell}}}&=&
\prod_{s=0}^{p^{n-2\ell}-1}\bigg(g_{(0,~\frac{1}{p^{n-\ell}m})\beta^s}^{12m}(\theta)\bigg)^{\alpha^{p^{n-2\ell}}}
=\prod_{s=0}^{p^{n-2\ell}-1}\bigg(g_{(0,~\frac{1}{p^{n-\ell}m})}^{12m}(\theta)\bigg)^{\beta^s{\alpha^{p^{n-2\ell}}}}\nonumber\\
&=&\prod_{s=0}^{p^{n-2\ell}-1}\bigg(g_{(0,~\frac{1}{p^{n-\ell}m})}^{12m}(\theta)\bigg)^{{\alpha^{p^{n-2\ell}}}\beta^s}
=\prod_{s=0}^{p^{n-2\ell}-1}\bigg(g_{(0,~\frac{1}{p^{n-\ell}m})}^{12m}(\theta)\bigg)^{{\alpha_1\alpha_2}\beta^s}\nonumber\\
&=&\prod_{s=0}^{p^{n-2\ell}-1}\bigg(g_{(0,~\frac{1}{p^{n-\ell}m})}^{12m}(\theta)\bigg)^{{\alpha_2}\beta^s}\nonumber\\
&&\textrm{because $g_{(0,~\frac{1}{p^{n-\ell}m})}^{12m}(\tau)$ has
Fourier coefficients in $\mathbb{Q}(\zeta_{p^{n-\ell}m})$
by definition (\ref{FourierSiegel})}\nonumber\\
&=&\prod_{s=0}^{p^{n-2\ell}-1}\bigg(g_{(0,~\frac{1}{p^{n-\ell}m})\alpha_2}^{12m}(\theta)\bigg)^{\beta^s}
=\prod_{s=0}^{p^{n-2\ell}-1}\bigg(g_{(p^{n-\ell}C',~\frac{E'}{p^{n-\ell}m}+p^{n-\ell}D')}^{12m}(\theta)\bigg)^{\beta^s}
\nonumber\\
&=&\prod_{s=0}^{p^{n-2\ell}-1}\bigg(g_{(0,~\frac{1}{p^{n-\ell}m})}^{12m}(\theta)\bigg)^{\beta^s}\nonumber\\
&&\textrm{by the fact $E'\equiv1\pmod{p^{n-\ell}m}$ and Proposition
\ref{transformation}(iv)}\nonumber\\
&=&\prod_{s=0}^{p^{n-2\ell}-1}g_{(0,~\frac{1}{p^{n-\ell}m})\beta^s}^{12m}(\theta)=g(\theta).
\end{eqnarray}
Observe that in particular, $g(\theta)$ is fixed by
$\alpha^{p^{n-\ell}}$ and hence by
$\langle\alpha^{p^{n-\ell}}\rangle
\times\langle\beta^{p^{n-\ell}}\rangle$. Thus $g(\theta)$ belongs to
$K_{(p^nm)}$ by (\ref{galois}).
\par
On the other hand, we see from (\ref{K_N/H}) that
$\mathrm{Gal}\big(K_{(p^nm)}/K_{(p^\ell
m)}\big)=\langle\alpha\rangle\times\langle\beta\rangle$ in
$\mathrm{GL}_2(\mathbb{Z}/p^nm\mathbb{Z})/\big\{\pm\left(\begin{smallmatrix}1&0\\0&1\end{smallmatrix}\right)\big\}$
with $\alpha$ and $\beta$ of order $p^{n-\ell}$. Suppose that
$\alpha^A\beta^B$ fixes both $\zeta_{p^n}$ and $g(\theta)$ for some
$0\leq A,~B<p^{n-\ell}$. Since
$(\zeta_{p^n})^{\alpha^A\beta^B}=(\zeta_{p^n})^{\det(\alpha^A)}=\zeta_{p^n}^{(1+p^\ell
m)^{2A}}$ by (\ref{K_N/H}), (\ref{first}) and (\ref{second}), we
have $A=0$. It then follows $B=0$ from (\ref{beta}). Therefore we
conclude that $K_{(p^nm)}=K_{(p^\ell
m)}\big(\zeta_{p^n},~g(\theta)\big)$ by Galois theory.
\end{proof}

Now we are in the following situation:
\begin{eqnarray*}
\begindc{\commdiag}
\obj(5,1)[A]{$K_{(p^\ell m)}$} \obj(2,3)[B]{$K_{(p^\ell
m)}(\zeta_{p^n})$} \obj(8,3)[C]{$K_{(p^\ell m)}\big(g(\theta)\big)$}
\obj(5,5)[D]{$K_{(p^nm)}$} \obj(7,5)[E]{$\quad=K_{(p^\ell
m)}\big(\zeta_{p^n},~g(\theta)\big)$} \mor{A}{B}{\tiny cyclic of
degree $p^{n-\ell}$}[\atleft,\solidline] \mor{A}{C}{\tiny cyclic of
degree $p^{n-\ell}$}[\atright,\solidline] \mor{B}{D}{\tiny cyclic of
degree $p^{n-\ell}$}[\atleft,\solidline] \mor{C}{D}{\tiny cyclic of
degree $p^{n-\ell}$}[\atright,\solidline]
\enddc
\end{eqnarray*}
Here we see $K_{(p^\ell m)}(\zeta_{p^n})\cap K_{(p^\ell
m)}\big(g(\theta)\big)=K_{(p^\ell m)}$ by analyzing the actions of
$\alpha$ and $\beta$ in the proof of Lemma \ref{K_pn}. Then we are
ready to attain our aim by means of the following two lemmas:

\begin{lemma}\label{help1}
Let $L$ be a number field and $F/L$ be a cyclic extension of a prime
power degree $n=p^s$ which is unramified outside $p$.
\begin{itemize}
\item[(i)] When $\zeta_n\in L$, $F$ has a normal $p$-integral
basis over $L$ if and only if $F=L(\sqrt[n]{\xi})$ for some $\xi\in
\mathcal{O}_L[\frac{1}{p}]^*$.
\item[(ii)] When $\zeta_n\not\in L$, $F$ has a normal $p$-integral
basis over $L$ if and only if $F(\zeta_n)$ has a normal $p$-integral
basis over $L(\zeta_n)$.
\end{itemize}
\end{lemma}
\begin{proof}
See \cite{Greither} Chapter 0 Proposition 6.5 and Chapter I Theorem
2.1.
\end{proof}

\begin{lemma}\label{help2}
Let $L$ be a number field and $F_s/L$ be a cyclic extension which is
unramfied outside $p$ for $s=1,2$. If $F_s$ has a normal
$p$-integral basis over $L$ for $s=1,~2$ and $F_1\cap F_2=L$, then
$F_1F_2$ has a normal $p$-integral basis over $L$.
\end{lemma}
\begin{proof}
See \cite{Kawamoto} p. 227.
\end{proof}

\begin{theorem}\label{Ko}
Let $K(\neq\mathbb{Q}(\sqrt{-1}),~\mathbb{Q}(\sqrt{-3}))$ be an
imaginary quadratic field, $p\geq5$ be a prime and $m\geq1$ be an
integer relatively prime to $p$. And, let $n$ and $\ell$ be positive
integers with $n\geq2\ell$. Then $K_{(p^nm)}$ has a normal
$p$-integral basis over $K_{(p^\ell m)}$.
\end{theorem}
\begin{proof}
The extension $K_{(p^\ell m)}(\zeta_{p^n})/K_{(p^\ell m)}$ is cyclic
of degree $p^{n-\ell}$ and is unramified outside $p$. We consider
the extension $K_{(p^\ell m)}(\zeta_{p^n})/K_{(p^\ell
m)}(\zeta_{p^{n-\ell}})$. It is also a cyclic extension of degree
$p^\ell$ unramified outside $p$ by considering the action of
$\alpha$ on $\zeta_{p^n}$. Since $K_{(p^\ell
m)}(\zeta_{p^n})=K_{(p^\ell
m)}(\zeta_{p^{n-\ell}})(\sqrt[p^l]{\zeta_{p^{n-\ell}}})$,
$K_{(p^\ell m)}(\zeta_{p^n})$ has a normal $p$-integral basis over
$K_{(p^\ell m)}(\zeta_{p^{n-\ell}})$ by Lemma \ref{help1}(i). If
$n=2\ell$, then $\zeta_{p^{n-\ell}}=\zeta_{p^\ell}$ and $K_{(p^\ell
m)}=K_{(p^\ell m)}(\zeta_{p^{n-\ell}})$. If $n>2\ell$, then
$\zeta_{p^{n-\ell}}\not\in K_{(p^\ell m)}$ and hence $K_{(p^\ell
m)}(\zeta_{p^n})$ has a normal $p$-integral basis over $K_{(p^\ell
m)}$ by Lemma \ref{help1}(ii).
\par
Next we look into the cyclic extension $K_{(p^\ell
m)}\big(g(\theta)\big)/K_{(p^\ell m)}$ of degree $p^{n-\ell}$
unramified outside $p$. Let us examine the extension $K_{(p^\ell
m)}\big(\zeta_{p^{n-\ell}},~g(\theta)\big)/K_{(p^\ell
m)}(\zeta_{p^{n-\ell}})$. Then we see that it is also a cyclic
extension of degree $p^{n-\ell}$ unramified outside $p$ by
considering the action of $\beta$ on $g(\theta)$. Note that we can
rewrite it as $K_{(p^\ell
m)}\big(\zeta_{p^{n-\ell}},~g(\theta)\big)=K_{(p^\ell
m)}(\zeta_{p^{n-\ell}})(\sqrt[p^{n-\ell}]{g(\theta)^{p^{n-\ell}}})$.
On the other hand, we know by (\ref{K_N/H}), (\ref{first}) and
(\ref{second}) that
$\mathrm{Gal}\big(K_{(p^{2(n-\ell)}m)}/K_{(p^\ell
m)}(\zeta_{p^{n-\ell}})\big)=\langle\alpha^{p^{n-2\ell}}\rangle\times\langle\beta\rangle$.
So $g(\theta)^{p^{n-\ell}}$ belongs to $K_{(p^\ell
m)}(\zeta_{p^{n-\ell}})$ by (\ref{beta}) and (\ref{alpha}).
Moreover, it belongs to $\mathcal{O}_{K_{(p^\ell
m)}(\zeta_{p^{n-\ell}})}[\frac{1}{p}]^*$ by Proposition
\ref{transformation}(i) and (ii). Hence $K_{(p^\ell
m)}\big(\zeta_{p^{n-\ell}},~g(\theta)\big)$ has a normal
$p$-integral basis over $K_{(p^\ell m)}(\zeta_{p^{n-\ell}})$ by
Lemma \ref{help1}(i). If $n=2\ell$, then $K_{(p^\ell m)}=K_{(p^\ell
m)}(\zeta_{p^{n-\ell}})$. If $n>2\ell$, then
$\zeta_{p^{n-\ell}}\not\in K_{(p^\ell m)}$ and hence $K_{(p^\ell
m)}\big(g(\theta)\big)$ has a normal $p$-integral basis over
$K_{(p^\ell m)}$ by Lemma \ref{help1}(ii).
\par
Therefore the theorem follows from Lemma \ref{help2} because
$K_{(p^\ell m)}(\zeta_{p^n})\cap K_{(p^\ell
m)}\big(g(\theta)\big)=K_{(p^\ell m)}$.
\end{proof}

\begin{corollary}
Let $K(\neq\mathbb{Q}(\sqrt{-1}),~\mathbb{Q}(\sqrt{-3}))$ be an
imaginary quadratic field, $p\geq5$ be a prime and $m\geq1$ be an
integer relatively prime to $p$. Let $K_\infty$ be any
$\mathbb{Z}_p$-extension of $K$. Then the $\mathbb{Z}_p$-extension
$K_\infty K_{(p^\ell m)}$ has a normal basis over $K_{(p^\ell m)}$
for $\ell\geq1$.
\end{corollary}
\begin{proof}
It is a direct consequence of Theorem \ref{Ko} by the well-known
fact that if an extension $F/L$ of number fields has a normal
$p$-integral basis over $L$, then so does $F'/L$ for each
intermediate field $F'$.
\end{proof}

\bibliographystyle{amsplain}

\end{document}